\documentclass[12pt,psamsfonts]{amsart}

\usepackage{amsmath,amssymb}
\usepackage{graphicx}
\usepackage{color}
\usepackage{psfrag} % See the section on figures.
\usepackage{overpic} % See the section on figures.

\newtheorem{theorem}{Theorem}
\newtheorem{proposition}[theorem]{Proposition}
\newtheorem{corollary}[theorem]{Corollary}
\newtheorem{lemma}[theorem]{Lemma}

\newtheorem{definition}[theorem]{Definition}

\newtheorem{observation}[theorem]{Observation}

\newcommand{\bR}{{\bf R}}

\title[Monodromy and Dulac's Problem for PWAVF]{
Monodromy and Dulac's Problem for Piecewise analytical planar vector fields}
\author[C. Buzzi, J. Medrado and C. Pessoa]{}

  \subjclass[2010]{34C25, 34C05, 34C07}
   \keywords{Piecewise analytical planar vector fields, monodromic singular point, Poincar\'e map, Lyapunov coefficients}

\begin{document}
 \maketitle

\centerline{\scshape  Claudio A. Buzzi,  \; Claudio Pessoa}
\medskip

{\footnotesize \centerline{Universidade Estadual Paulista (UNESP), Instituto de Bioci\^encias Letras e Ci\^encias Exatas,} \centerline{R. Cristov\~ao Colombo, 2265, 15.054-000, S. J. Rio Preto, SP, Brasil }
\centerline{\email{claudio.buzzi@unesp.br} and \email{c.pessoa@unesp.br}}}

\medskip

\centerline{\scshape Jo\~ao C. Medrado}
\medskip

{\footnotesize \centerline{Universidade Federal de Goi\'as, Instituto de Matem\'atica e Estat\'\i stica, } \centerline{Campus
Samambaia,74001-970, Goi\^ania, GO, Brasil }
\centerline{\email{medrado@ufg.br}}}

\medskip

\bigskip

\begin{quote}{\normalfont\fontsize{8}{10}\selectfont
{\bfseries Abstract.}
Consider an analytical function $f:V\subset\bR^2\rightarrow\bR$ having $0$ as its regular value, a switching manifold $\Sigma=f^{-1}(0)$  and a piecewise analytical vector field $X=(X^+,X^-)$, i.e. $X^\pm$ are  analytical vector fields defined on $\Sigma^\pm=\{p\in V: \pm f(p)>0\}$. We characterize when the vector field $X$ has a monodromic singular point in $\Sigma$, called $\Sigma$-monodromic singular point. Moreover, under certain conditions, we show that a $\Sigma$-monodromic singular point of $X$ has a neighborhood free of limit cycles.
\par}
\end{quote}

\section{Introduction}
%\label{sec:01}

The finiteness problem for the number of limit cycles for analytical vector fields is known as Dulac's problem. In 1923, H. Dulac \cite{Dulac} proved that every real planar polynomial vector field has only a finite number of limit cycles. But, unfortunately, some errors have been found in his proof in the 1980s. For analytical vector fields on the plane, it was solved independently by J. \'Ecalle \cite{Ecalle} and Yu. S. Ilyashenko \cite{Ilyashenko}.

We deal with piecewise analytical vector fields $X$ (in short, PWAVF) defined by a pair $(X^+, X^-)$, where $X^+$ and $X^-$ are analytical vector fields on regions of the plane separated by an analytical curve $\Sigma$. The curve $\Sigma$ is called \textit{switching curve} which separates the plane in two regions $\Sigma^+$ and $\Sigma^-$ having defined on each region  the analytical vector fields $X^+$ and $X^-$, respectively.  In \cite {F}, Filippov defined the rules (revisited below)  for the transition of orbits crossing the switching analytical curve $\Sigma$.  He also describes when an orbit slides along $\Sigma$. This leads to an orbiting structure that is not always a flow on the plane obtained gluing the orbits on $\Sigma^+$  and $\Sigma^-$  along $\Sigma$. 

As far as we know, there is no systematic study of monodromic singular points of piecewise analytical planar vector fields belonging to $\Sigma$.  We call these points $\Sigma$-monodromic singular points. In fact, the concept of monodromy is not formalized for these types of vector fields. The few works that deal with this topic study the so-called ``pseudo focus" or ``sewed focus" (see \cite{CGP, F}). In the literature, four types of pseudo focus $p\in\Sigma$ of piecewise analytical planar vector fields $X=(X^+, X^-)$ are treated.
\begin{itemize}
	\item[(i)] {\it Focus-Focus type}: both systems $X^+$ and $X^-$ have a singular point at $p$ with eigenvalues $\lambda\in\mathbb{C} \setminus \mathbb{R}$ and their solutions turn around $p$ counterclockwise or clockwise sense.
	\item[(ii)] {\it Focus-Parabolic}  (resp. {\it Parabolic-Focus}) : $X^+$ (resp. $X^-$) has a singular point of focus type at p (i.e. $X^+$ has a singular point at $p$ with  eigenvalues $\lambda\in\mathbb{C} \setminus \mathbb{R}$) while the solutions of $X^-$ (resp. $X^+$) have a parabolic contact (i.e. a second order contact point) with $\Sigma$ at p, the solution by $p$ is locally contained in $\Sigma^+$ (resp. $\Sigma^-$). 
	\item[(iii)] {\it Parabolic-Parabolic}: the solutions of both systems have a parabolic contact at $p$ with $\Sigma$ in such a way that the flow of $X$ turns around p.
\end{itemize}

In \cite{CGP} the authors introduce techniques that make it possible to study the stability of the monodromic singular points described in the statements (i), (ii), and (iii). Moreover, they obtain general expressions for the first three Lyapunov constants for these points and generate limit cycles for some concrete examples. In 2021, in the work \cite{NS2021}, the authors consider the case Parabolic-Parabolic having folds with contact order greater or equal to two. We observe that the results about analyticity of the return maps in \cite{CGP} and  \cite{NS2021} are contained in our work. We remark that our approach is distinct of \cite{NS2021} and the results were obtained in an independent way.  Bifurcations of limit cycles from the pseudo focus,  defined above, were studied in several papers; see for example \cite{ BPT, F, FPT, HZ, MT, ZKB}. The Center-Focus Problem for piecewise linear systems is solved in \cite{HZ}.  Other papers that deal with the Center-Focus Problem for particular families of piecewise planar vector fields with a pseudo focus are \cite{BPT, P}.

%\medskip

We work with two problems. The first one is the characterization of the $\Sigma$-monodromic singular points to piecewise analytical planar vector fields. The last one is Dulac's problem restricted to piecewise analytical planar vector fields, i.e., the existence of a neighborhood of the $\Sigma$-monodromic singular point (whenever possible) ``free of limit cycles".  As far as we know, the unique work that studies Dulac's problem for piecewise analytical planar vector fields is \cite{AJMT}.  In this paper, the authors consider polycycles whose the singular points outside $\Sigma$ are hyperbolic saddles. The singular points on $\Sigma$  are either hyperbolic saddles or regular points or fold points of order $2$ from each side of $\Sigma$. They prove that this type of polycycle does not have limit cycles accumulating onto it.

Our main results are the following. Theorem~\ref{the:01} characterizes all $\Sigma$-monodromic singular points of piecewise analytical vector fields. In true, this result also holds with a similar proof for piecewise smooth vector fields. Considering piecewise analytical vector field, in the Sections~\ref{sec:PRM1}, we prove that if $p\in \Sigma$ is a monodromic singular point, then the Poincar\'e return map is well defined (Theorem \ref{teo:mon_ret_map}). In the Section~\ref{sec:PRM2}, we give conditions such that the Poincar\'e return map is analytic (Theorem \ref{fund_lemma}), and consequently, there is a neighborhood of the $\Sigma$-monodromic singular point free of limit cycles (Corollary~\ref{cor:PRMA} and Remark \ref{obnil}). We conclude this work by applying our results to classes of piecewise analytical vector fields with a $\Sigma$-monodromic singular point that is not one of the pseudo foci described in the statements (i), (ii), and (iii) above (Section~\ref{sec:App}). 

\section{Piecewise analytical vector fields}
Let $Y:\mathbb R^2\rightarrow \mathbb R^2$ be an analytical vector field and $p$ a singular point of $Y$.  We recall the definitions involving singular points of $Y$. The singular point $p$ is called {\it non-degenerate} if $0$ is not an eigenvalue of $DY(p)$. The singular point $p$ is called {\it hyperbolic} if the two eigenvalues of $DY(p)$ have real part different from $0$. The singular point $p$ is called {\it semi-hyperbolic} if exactly one eigenvalue of $DY(p)$ is equal to $0$. Hyperbolic and semi-hyperbolic singular points are also said to be {\it elementary} singular points. The singular point $p$ is called {\it nilpotent} if both eigenvalues of $DY(p)$ are equal to $0$ but $DY(p) \not\equiv 0$. The singular point $p$ is called {\it linearly zero} if $DY(p) \equiv 0$. It is called {\it center} if there is an open neighborhood consisting, besides the singular point, of periodic orbits. The singular point is said to be {\it linearly center} if the eigenvalues of $DY(p)$ are purely imaginary without being zero.

Let $p$ be a singular point of $Y$ and $\gamma=\{\gamma(t):t\in\mathbb R\}$ be an orbit of $Y$ that tends to $p$ when $t$ tends to $\infty$. In this case,  $\gamma$ is a {\it characteristic orbit} if there exists $\lim_{t\rightarrow \infty} \gamma (t)/|\gamma (t)|$, where $|\cdot|$ is the Euclidean norm. We may also consider orbits tending to $p$ as t tends to $-\infty$, but we can always change the sign of the parameter $t$ in the system associated with $Y$ and assume that $t$ tends to $\infty$. We have that $p$ is a {\it monodromic singular point} of $Y$ if there is no characteristic orbit associated with it.

Consider $f:\mathbb R^2 {\rightarrow} {\mathbb R}$  an analytical function having $0$ as a regular value and  denote $\Sigma=f^{-1}(0)$ and $\Sigma^\pm=\{p\in\mathbb R^2: \pm f(p)>0\}$. Let $X=(X^+,X^-)$ be a piecewise analytical vector field defined by
\[
 X(q) = \left\{ \begin{array}{c}             X^+(q)
\hspace{0.3cm}   \mbox{if}   \hspace{0.3cm}   f(q)   \geq    0,    \\
X^-(q)   \hspace{0.3cm}   \mbox{if}   \hspace{0.3cm}   f(q)   \leq
0,
\end{array}\right.
\]
where $X^\pm$ are analytical planar vector fields. We note that $X$ can be bi-valued at the points of $\Sigma$. Following
Filippov's terminology (in \cite {F}), we distinguish the
following arcs in $\Sigma$:
\begin{itemize}
\item[i)] Sewing  arc  $(SW)$:  characterized  by  $X^+f\cdot X^-f  >
0$ (see Figure~\ref{fig:01} (a)). 
\item[ii)] Escaping arc $(ES)$:
given by the inequalities $X^+f>0$ and $X^-f<0$ (see
Figure~\ref{fig:01} (b)). 
\item[iii)] Sliding arc $(SL)$: given by
the inequalities $X^+f<0$ and $X^-f>0$ (see Figure~\ref{fig:01} (c)).
\end{itemize}

As usual, here and in what follows, $Yf$ will denote the
derivative of the function $f$  in the direction of the vector
$Y$, i.e. $Yf=\langle \nabla f, Y\rangle$. Moreover, $Y^nf=Y(Y^{n-1}f)$.

\bigskip

\begin{figure}[ht]
	\begin{center}		
		\begin{overpic}[width=0.95\textwidth]{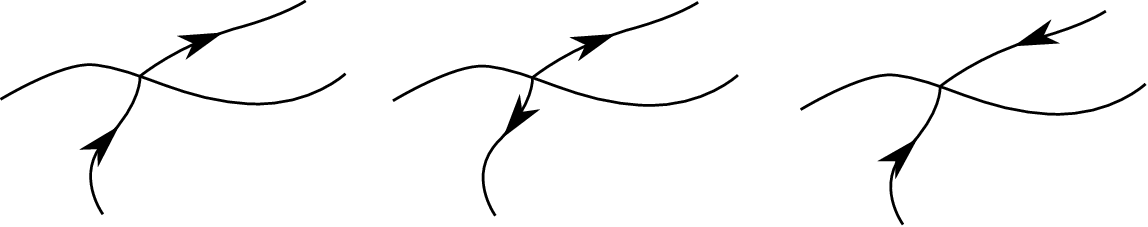}
%  	\begin{overpic}[grid,width=0.95\textwidth]{types_pseudo2.eps}
			\put(12,-5) {$(a)$}
			\put(47,-5) {$(b)$}
			\put(83,-5){$(c)$}
			\put(12,10) {$p$}
			\put(47,10) {$p$}
			\put(82,9.5){$p$}
			\put(25,8){$\Sigma$}
			\put(60,8){$\Sigma$}
			\put(95,7){$\Sigma$}
		\end{overpic}
	\end{center}
	\vspace{0.7cm}
	\caption{Illustrating different types of $p\in \Sigma$: $(a)$ Sewing; $(b)$ Escaping; $(c)$ Sliding.}
	\label{fig:01}
\end{figure}
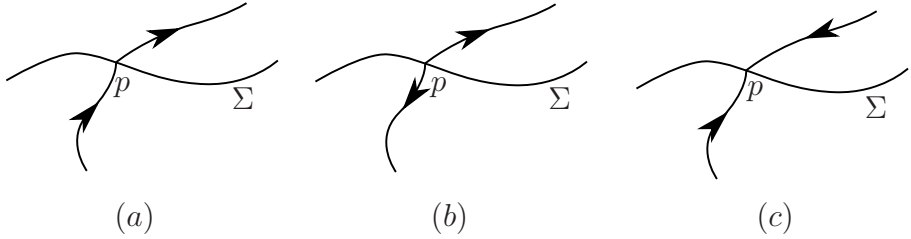  

On the arcs $ES$ and $SL$ we define the {\it Filippov vector
field} $F_{X}$ associated  to $X = (X^+,X^-)$, as follows: if $p \in
SL$ or $ES$, then $F_{X}(p)$  denotes  the vector tangent to $\Sigma$
in the cone spanned by $X^+(p)$  and $X^-(p)$. A  point  $p  \in  \Sigma$  is  called a {\it $\Sigma$-regular
point} of $X$ if $p\in SW$ or if $p\in SL$ or $p\in ES$ then $F_X(
p)\neq 0$. Otherwise, $p$ is a {\it $\Sigma$-singular point} of $X$.

Following \cite{ST}, a continuous closed curve $\gamma$ consisting of two regular
trajectories, one of $X^+$ and another of $X^-$, and two points
$\{p_1,p_2\}=\Sigma\cap\gamma$ is called a {\it $\Sigma$-closed crossing orbit } (or simply {\it $\Sigma$-closed orbit}) of $X$, if
$\{p_1,p_2\}$ are sewing points and $\gamma$ meets $\Sigma$
transversally in $\{p_1,p_2\}$.

\begin{definition}
\label{def_ChO}
Let $p$ be an isolated $\Sigma$-singular point of a piecewise analytical
vector field $X=(X^+,X^-)$. We say that $p$ has a $\Sigma$-{\it
characteristic orbit} if one of the following conditions  hold:
\begin{itemize}
\item[i)]  There exists a regular trajectory $\gamma$ of $X^+$ (resp.
$X^-$), with $p=\gamma(t_0)$ for some $t_0\in\mathbb{R}$, and $p\in
\overline{\gamma\cap \Sigma^+}$ (resp. $p\in
\overline{\gamma\cap \Sigma^-}$);

\item[ii)] There exists a regular trajectory $\gamma$ of $X^+$ (resp.
$X^-$) with $\lim_{t\rightarrow\pm\infty}\gamma(t)=p$, and there
exist a neighborhood $V$ of $p$ such that $\gamma \cap V\subset
V\cap\Sigma^+$ (resp. $\gamma \cap V\subset V\cap
\Sigma^-$);

\item[iii)] For all neighborhood $V$ of $p$ there exists $q\in V\cap\Sigma$ such that
$Xf(q)\cdot Yf(q) < 0$.
\end{itemize}
\end{definition}

If $p$ is an isolated $\Sigma$-singular point of $X$ and $X$ does not have
$\Sigma$-characteristic orbits associated with $p$ we call it $\Sigma$-{\it monodromic singular point}. 
\begin{definition}
\label{def_Mono}
A $\Sigma$-singular point $p$ of $X=(X^+,X^-)$ is a 
\begin{itemize}
\item[a)] {\it center} if there exist a
neighborhood $V$ of $p$ such that $V\setminus \{p\}$ is fulfilled of $\Sigma$-closed orbits of $X$;
\item[b)] {\it focus} if there exists a neighborhood $V$ of $p$ such that for all orbits $\gamma$ of $X$ by points in $V$  spirals toward or backward $p$;
\item[c)] {\it center-focus} if there exists a sequence of  \; $\Sigma$-closed orbits $\gamma_n$ of $X$, with $\gamma_{n+1}$ in the interior of $\gamma_n$ such that $\gamma_n\rightarrow p$ as $n\rightarrow \infty$ and such that every trajectory between $\gamma_n$ and $\gamma_{n+1}$ spirals toward $\gamma_n$ or $\gamma_{n+1}$ as $t$ increase or decrease. 
\end{itemize}
\end{definition}
By Definitions \ref{def_ChO} and \ref{def_Mono}, a $\Sigma$-singular point is a $\Sigma$-monodromic singular point if and only if it is a center or focus or center-focus.
 
A $\Sigma$-singular point $p$ is a \textit{fold$_{n^\pm}$} of $X^\pm$ if it is a fold  point of order $n^\pm$, i.e.,
\[
X^\pm(p)\neq 0, X^\pm f(p) =\cdots =(X^\pm)^{n^\pm -1}f({p})= 0 \mbox{ and } (X^\pm)^{n^\pm} f(p)\neq 0,
\]
where  $n^\pm$ are even. In this case we have that $X^\pm$ has a contact of order $n^\pm$ with $\Sigma$ at $p$ (see Figure~\ref{fig:02}). 

We say that $p$ is a {\it visible fold$_{n^\pm}$} for $X^\pm$ if $\pm (X^\pm)^{n^\pm}f(p)>0$. If $\pm (X^\pm)^{n^\pm}f(p)<0$ then we say that $p$ is an {\it invisible fold$_{n^\pm}$} for $X^\pm$.
\begin{figure}[ht]
	\begin{center}		
		\begin{overpic}[width=0.95\textwidth]{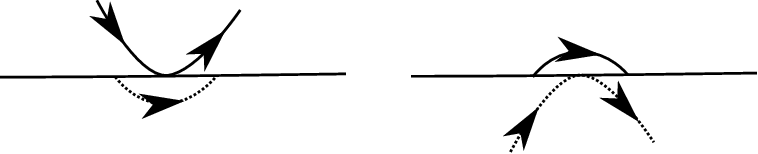}
		%\begin{overpic}[grid,width=0.95\textwidth]{folds.eps}
			\put(20,-5) {$(a)$}
			\put(75,-5) {$(b)$}
			\put(21.5,12) {$p$}
			\put(75.5,7.5) {$p$}
			\put(40,7){$\Sigma$}
			\put(95,7){$\Sigma$}
		\end{overpic}
	\end{center}
	\vspace{0.7cm}
	\caption{Illustrating different types of folds: $(a)$ Visible fold$_{n^+}$ for $X^+$; $(b)$ Invisible fold$_{n^+}$ for $X^+$.}
	\label{fig:02}
\end{figure}
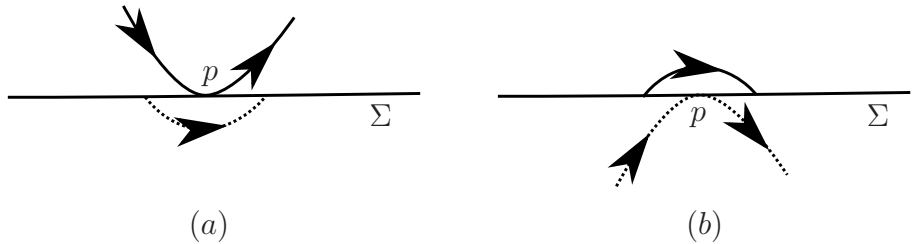 

\section{$\Sigma$-Monodromic singular points}

In this section, we give necessary and sufficient conditions to a
$\Sigma$-singular point be a $\Sigma$-monodromic singular point.

\begin{proposition}
	\label{pro:02} 
	Let $p$ be an isolated $\Sigma$-singular point of $X=(X^+,X^-)$. If
	$p$ is a $\Sigma$-monodromic singular point, then $X^+f(p)=X^-f(p)=0$ and
	there is a neighborhood $V$ of $p$ in $\Sigma$  with $X^+f\cdot X^-f\mid_V
	>0$ with the unique exception of $p$.
\end{proposition}
\begin{proof}
	As $p$ is a $\Sigma$-singular point of $X=(X^+,X^-)$, we have that  $X^+f({p})\cdot X^-f({p})\leq 0$.  In the first case $X^+f({p})\cdot X^-f({p})< 0$, i.e. $p\in SL\cup ES$ and $F_X({p})=0$, there exist a neighborhood $V$ of $p$ such that $X^+f\cdot X^-f\mid_V
	<0$. Hence, by statement {\rm (iii)} of Definition \ref{def_ChO}, $p$ has a $\Sigma$-characteristic orbit. In the second case we have $X^+f({p})\cdot X^-f({p})= 0$. Note that if $X^+f({p})\neq 0$ (resp. $X^-f({p})\neq 0$), then the flow of $X^+$ (resp. $X^-$) is transversal to $\Sigma$ at $p$ and so, by statement {\rm (i)}, $p$ has a $\Sigma$-characteristic orbit. Therefore $X^+f(p)=X^-f(p)=0$. 
	
	Assume that does not exist a neighborhood $V$ of $p$ in $\Sigma$ such that  $X^+f\cdot X^-f\mid_{V\setminus\{p\}}
	>0$, i.e., for all neighborhood $V$ of $p$ in $\Sigma$ there exists $q\in V$ such that $X^+f({q})\cdot X^-f({q})< 0$. So, by statement {\rm (iii)} of Definition \ref{def_ChO}, $p$ has a $\Sigma$-characteristic orbit. 
\end{proof}

\begin{observation}
From Proposition \ref{pro:02} it is clear that if $p$ is a $\Sigma$-mo\-no\-dro\-mic singular point then the orientation of the orbits of vector fields $X^+$ and $X^-$ around the singularity agrees in such way that the trajectories of both vector fields can be concatenated at $\Sigma$ in suitable order.
\end{observation}

To state the next result, we need the following definition. Let $p\in \overline{K}\subset\mathbb R^2$ an isolated singular point of an analytical vector field $Y$.  We say that $Y$  has a {\it characteristic orbit in} $K$ associated with $p$ if there are a characteristic orbit $\gamma$ of $Y$ associated with $p$ and a neighborhood $V$ of $p$ such that $\gamma \cap V\subset V\cap K$. Otherwise, we say that $Y$ has no characteristic orbit in $K$. %We say that $Y$ {\it is monodromic in} $K$ if $Y$ does not have a characteristic orbit in $K$.
The following theorem classifies the $\Sigma$-monodromic singular
points of $X$.

\begin{theorem}
\label{the:01}
Let $X=(X^+,X^-)$ be a piecewise analytical vector field. We
suppose that $p$ is a $\Sigma$-singular point of $X$ such that
$X^+f(p)=X^-f(p)=0$ and $X^+f\cdot X^-f\mid_V >0$ in a $V$-neighborhood of $p\in \Sigma$ with the unique exception of $p$.
\begin{itemize}
\item[i)] If $X^+(p)\neq 0$ and $X^-(p)\neq 0$, then $p$ is a
$\Sigma$-monodromic singular point of $X$ if and only if it is a fold
of order $n^+$ of $X^+$ with $(X^+)^{n^+}f(p)<0$ and it is a fold of order
$n^-$ of $X^-$ with $(X^-)^{n^-}f(p)>0$.

\item[ii)] If $X^+(p)=0$ (resp. $X^-(p)=0$) and $X^-(p)\neq 0$
(resp. $X^+(p)\neq 0$), then $p$ is a $\Sigma$-monodromic singular
point of $X$ if and only if $X^+$ has no characteristic orbit in $\Sigma^+$
(resp. $X^-$  has no characteristic orbit in $\Sigma^-$) associated with $p$ and $p$ is a fold of order $n$ of $X^-$
(resp. $X^+$) with $(X^-)^nf(p)>0$ (resp. with
$(X^+)^nf(p)<0$).

\item[iii)] If $X^+(p)=X^-(p)=0$, then $p$ is a $\Sigma$-monodromic
singular point of $X$ if and only if $X^+$ has no characteristic orbit in $\Sigma^+$ and $X^-$ has no characteristic orbit in $\Sigma^-$ associated with $p$, respectively.
\end{itemize}
\end{theorem}
\begin{proof}
As $p$ is a $\Sigma$-singular point of $X$ such that $X^+f(p)=X^-f(p)=0$
and $X^+f\cdot X^-f\mid_V >0$ in a $V$-neighborhood of $p$ in $\Sigma$ with the
unique exception of $p$, it follows that $p$ is an isolated
singular point of $X$. Now, $p$ is a $\Sigma$-monodromic singular point
of $X$ if and only if  it is isolated and $X$ has no
$\Sigma$-characteristic orbit associated with $p$. Therefore, the hypothesis and
conditions stated in each case ((i), (ii), and (iii)) are necessary
and sufficient to the nonexistence of $\Sigma$-characteristic orbits of $X$
associated with $p$. This finishes the proof.
\end{proof}

\section{Existence of Poincar\'e return map}
\label{sec:PRM1}

In the sequel, we  consider  $f:\mathbb R^2 {\rightarrow} {\mathbb R}$   an analytical function having $0$ as a regular value and  denote $\Sigma=f^{-1}(0)$ and $\Sigma^\pm=\{p\in\mathbb R^2:\pm f(p)>0\}$.  Also, we consider $p$ a $\Sigma$-singular point of $X=(X^+,X^-)$ such that the analytical vector field $X^+$ ($X^-$) does not have characteristic orbit in $\Sigma^+$ ($\Sigma^-$) associated with $p$. In this case we have the following situation: given $q^+\in\Sigma$ and an orbit $\varphi^+$ of $X^+$ such that $\varphi^+(0, q^+)=q^+$,  we can determine flight time $\tau^+$ for which $q^-=\varphi^+(\tau^+, q^+)\in \Sigma$ and $\varphi^+ \subset \Sigma^+$ for all time $t$ between $0$ and $\tau^+$. So, it makes sense the study of the application $\Pi^+: V \rightarrow W$ where $q^+\in V\subset\Sigma$ and  $\varphi^+(\tau^+, q^+) \in W \subset \Sigma$. We call $\Pi^+$ the \textit{half-Poincar\'e return map.}  In a similar way, we define $\Pi^-$ considering an orbit $\varphi^- $ of $X^-$ such that $\varphi^-(0, q^-)=q^-\in W\subset \Sigma$ and $\varphi^-(\tau^-, q^-) \in V \subset \Sigma$ with $\varphi^-\subset \Sigma^-$ for all time $t$ between $0$ and $\tau^-$. We also define the \textit{displacement function} given by $\Psi(q^-)=\Pi^+(\varphi^+(-\tau^+, q^-))-\Pi^-(q^-)$, or $\Psi=\Pi^+-(\Pi^-)^{-1}$. Observe that $(\Pi^-)^{-1}$ is the half-Poincar\'e return map $\Pi^-$ associated with vector field $-X^{-}$. See Figure~\ref{poincare}.

\begin{figure}[ht]
	\label{poincare}
	\begin{center}		
%		\begin{overpic}[grid]{poincare.eps}
		\begin{overpic}{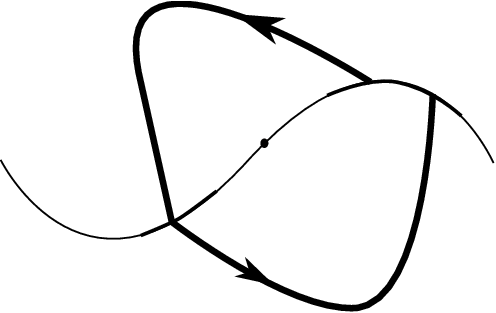}
			\put(73,42) {$q^+$}
			\put(5,19) {$\Pi^+(q^+)=q^-$}
			\put(77,48){$V$}
			\put(44,20) {$W$}
			\put(88,44){$\Pi^-(q^-)$}
			\put(80,0){$\varphi^-$}
			\put(24,60){$\varphi^+$}
			\put(2,30){$\Sigma$}
			\put(53,28){$p$}
		\end{overpic}
	\end{center}
	\caption{The Poincar\'e return map $\Pi(q^+)=\Pi^-(\Pi^+(q^+)$ where $q^+\in V$ and $ q^-\in W$ such that $q^-=\varphi^+(\tau^+,q^+)=\Pi^+(q^+)$ and the displacement function $\Psi(q^-)=\Pi^+(\varphi^+(-\tau^+, q^-))-\Pi^-(q^-)$.}
%	\label{Ilipsbeak}
\end{figure}
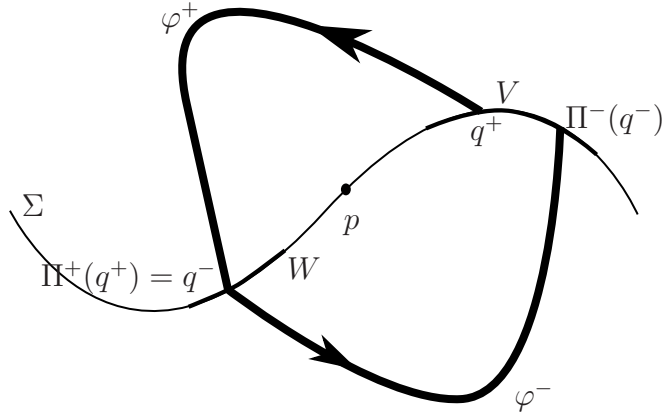  
Now, we will study under which conditions the Poincar\'e return map $\Pi=\Pi^-\circ\Pi^+$ is well defined on $\Sigma$.

\begin{lemma}\label{prop5}
Let $Y$ be an analytical vector field in some open subset $U$ of  $\mathbb R^2$ %that contains  $p\in\Sigma$ 
and $\Sigma=f^{-1}(0)$, where $f:U\subset\mathbb R^2\rightarrow\mathbb R$ is analytic and $0$ is its regular value. The following statements hold.
\begin{itemize}
	\item[i)]  If $p\in\Sigma$ is an isolated singular point  and $Y$ does not have characteristic orbit in $\Sigma^+\cup\Sigma$ (resp. $\Sigma^-\cup\Sigma$) associated with $p$, then the half-Poincar\'e return map  is well defined in a neighborhood $V\cap(\Sigma^+\cup\Sigma)$ of $p$  (resp. $V\cap(\Sigma^-\cup\Sigma)$).
	\item [ii)] If $p\in\Sigma$ is an invisible (resp. visible) fold$_n$ of $Y,$ then the half-Poincar\'e return map  is well defined in a neighborhood $V\cap(\Sigma^+\cup\Sigma)$ of $p$ (resp. $V\cap(\Sigma^-\cup\Sigma)$).
\end{itemize}

\end{lemma}

\begin{proof}
Firstly, we suppose that $p$ is a singular point of $Y$. We deal with the case that $Y$ does not have a characteristic orbit in $\Sigma^+\cup\Sigma$ associated with $p$. The other case is analogous.  It is well known that if $p$ is a singular point of an analytical vector field $Y$ then any trajectory of $Y$ which approaches the singular point $p$ as $t\rightarrow\pm\infty$ either spirals toward to $p$ or is a characteristic orbit of $Y$ (see \cite[Theorem 64, p. 331]{Andronov}).  The hypothesis that $Y$ does not have characteristic orbit in $\Sigma^+\cup\Sigma$ associated with $p$ implies that either $p$ is a monodromic singular point of $Y$ or the characteristic orbits of $Y$ are contained in $\Sigma^-$. In the first case, it is obvious that the half-Poincar\'e return map is well defined in $V\cap(\Sigma^+\cup\Sigma)$, for some neighborhood $V$ of $p$. In the second one, we recall that any neighborhood $V$ of an isolated singular point $p$ of an analytical vector field $Y$ is a finite union of elliptic, hyperbolic, or parabolic sectors (see  \cite{Andronov}). Thus, as the borders of sectors of singular points are characteristic orbits, it follows that $V\cap(\Sigma^+\cup\Sigma)$ is contained either in an elliptic or a hyperbolic sector of $p$. So, the half-Poincar\'e return map is well defined in $V\cap(\Sigma^+\cup\Sigma)$ for both cases. 

If $p$ is an invisible fold$_n$ of $Y$ and $Y$ then the orbit of $Y$ by $p$ is contained in $\Sigma^-$. Therefore the half-Poincar\'e return map  is well defined in $V\cap(\Sigma^+\cup\Sigma)$. The case where $p$ is a visible fold$_n$ of $Y$ is analogous.
\end{proof}

\begin{observation} If $p$ is  a $\Sigma$-singular point of $X=(X^+,X^-)$ with $X^
	+ (p)=0$  (resp. $X^-(p)=0$) and $p$ is an elementary singular point of $X^+$  (resp. $X^-$) with real eigenvalues, then by the proof of Lemma \ref{prop5} we conclude that it cannot be a $\Sigma$-monodromic singular point. Because, in this case, $X^+$  (resp. $X^-$) has characteristic orbits associated with $p$ in $\Sigma^+\cup \Sigma$ and $\Sigma^-\cup \Sigma$.
\end{observation}

In the classic theory of analytical vector fields, we have that monodromy implies that the Poincar\'e return map is well defined. The next theorem shows that we have a similar result for piecewise analytical vector fields.

\begin{theorem}\label{teo:mon_ret_map}
	Let $p$ be an isolated $\Sigma$-singular point of a piecewise analytical vector field  $X=(X^+, X^-)$. If 	$p$ is a $\Sigma$-monodromic singular point, then the Poincar\'e return map is well defined.
\end{theorem}
\begin{proof}
	As $p$ is a  $\Sigma$-monodromic singular point, it follows that $X^+$ (resp. $X^-$ ) satisfies the hypothesis of Lemma \ref{prop5} in $\Sigma^+\cup\Sigma$ (resp.  $\Sigma^-\cup\Sigma$) with respect to $p$.  Then the half-Poincar\'e return maps are well defined. Moreover, there exists a neighborhood $V$ such that  $X^+f\cdot X^-f\mid_{V\setminus\{p\}}>0$. Therefore the Poincar\'e return map is obtained by the composition of these half-Poincar\'e return maps, i.e., the Poincar\'e return map is well defined.
\end{proof}

\section{Analyticity of the displacement function}
\label{sec:PRM2}

According to Theorem \ref{teo:mon_ret_map} monodromy implies existence of the Poincar\'e return map. Our concern now is with the analyticity of the displacement function. We will see that in the neighborhood of certain types of $\Sigma$-monodromic singular points, the displacement function is analytic.

\medskip

\noindent {\bf Hypothesis H}. {\it 
		Let  $X=(X^+,X^-)$ be a piecewise analytical vector field having an isolated $\Sigma$-monodromic singular point $p_0\in\Sigma$. Without loss of generality we assume that $p_0$ is the origin and the orbits of $X$ turn around the origin in counterclockwise sense. We say that $X$ satisfies the Hypothesis {\bf H} if 
		\begin{itemize}
			\item [(i)] there exists a  continuous change of coordinates such that the switching curve $\Sigma$, in these new coordinates, is  $\{y=0\}$; and
			\item [(ii)] there are integers $p,\ q,\ r$ and $s$ such that in the weight polar coordinates $(x,y)=(\rho^p\cos\theta,\rho^r\sin\theta)$ in $\Sigma\cup\Sigma^+$ and $(x,y)=(\rho^q\cos\theta,\rho^s\sin\theta)$ in $\Sigma\cup\Sigma^-$ the systems associated with vector fields $X^+$ and $X^-$ are equivalent to differential equations
			\begin{equation}\label{eqpolarpeso}
				\frac{d\rho}{d\theta}  =  \frac{F^\pm(\theta, \rho) }{G^{\pm}(\theta, \rho)},
			\end{equation}
			where $F^\pm$ and $G^\pm$ are analytical functions with $F^\pm (\theta, 0)=0$ for all $\theta\in \mathbb R$, $G^+(\theta, 0)\neq0$ for all $\theta\in[0,\pi]$ and $G^-(\theta, 0)\neq0$ for all $\theta\in[-\pi,0]$.
		\end{itemize}
	}

\medskip

Before we state and prove the main theorem of this section, we present classes of piecewise analytical vector fields that satisfy the Hypothesis~{\bf H}.  Theorem \ref{the:01} shows the three cases of monodromy. Next proposition deals with the first case, which is the case where the $\Sigma$-singular point is both fold points for $X^+$ and $X^-$.
	
\begin{proposition}
		\label{lemaFold}
		Let $X=(X^+,X^-)$ be a piecewise analytical  vector field and $\Sigma=f^{-1}(0)$ a curve, where $f:\mathbb R^2\rightarrow\mathbb R$ is  an analytical function having $0$ as its regular value.  Assume that $p$ is a fold point of order $n^{\pm}$ of $X^{\pm}$ with $(X^+)^{n^+}f(p)<0$, $(X^-)^{n^-}f(p)>0$ and  $X^+f\cdot X^-f\mid_{V\setminus\{p\}}>0$ for some neighborhood $V$ of $p$, then $X$ satisfies the Hypothesis~{\bf H}.
	\end{proposition}
	
\begin{proof}
		We can assume that $p=(0,0)$. By hypothesis $p$ is a fold point of order $n^{\pm}$ of $X^{\pm}$. This implies that, in local coordinates, we have $f(x,y)=y$ and the differential system associated with $X^{\pm}$ is given by
		\begin{equation}
			\label{sis1II}
			\begin{aligned}
				\dot{x} & = a_{00}^{\pm} + {\mathcal O}(|(x,y)|), \\
				\dot{y} & = b_{n^\pm-1,0}^{\pm}x^{n^\pm-1} + {\mathcal O}(x^{n^\pm})+y(b_{01}^{\pm}+b_{11}^{\pm}x+b_{02}^{\pm}y+{\mathcal O}(|(x,y)|^2)),
			\end{aligned}
		\end{equation}
		where $(X^{+})^{n^{+}}f(p)=(n^{+}-1)!\ a_{00}^{+} b_{n^+-1,0}^{+}<0$ and $(X^{-})^{n^{-}}f(p)=(n^{-}-1)!\ a_{00}^{-} b_{n^--1,0}^{-}>0$.  Doing rescaling of the time variable, if necessary, we can assume that the orbits of $X$ turn around of the origin in counterclockwise sense.  In order to verify Hypothesis~{\bf H} we consider the weight polar coordinates $(x,y)=(\rho \cos \theta, \rho^{n^+}\sin\theta)$ in $\Sigma\cup\Sigma^+$ and $(x,y)=(\rho \cos \theta, \rho^{n^-}\sin\theta)$ in $\Sigma\cup\Sigma^-$. So, using these new coordinates, system \eqref{sis1II} is equivalent to differential equation \eqref{eqpolarpeso} with 
		\[
		F^\pm(\theta,\rho)=( a_{00}^\pm\cos\theta  +b_{n^\pm-1,0}^\pm\sin\theta  \cos^{n^\pm-1}\theta )\rho+ {\mathcal O}(\rho^2)
		\] 
		and $G^\pm(\theta,\rho)=b_{n^\pm-1,0}^\pm \cos^{n^\pm}\theta -n^\pm a_{00}^\pm\sin\theta+ {\mathcal O}(\rho)$.
		 Note that the conditions $\pm a_{00}^\pm b_{n^\pm-1,0}^\pm<0$ imply $G^+(\theta, 0)\neq0$ for all $\theta\in[0,\pi]$ and $G^-(\theta, 0)\neq0$ for all $\theta\in[-\pi,0]$. So, $X$ satisfies the Hypothesis~{\bf H}.
\end{proof}

The next result exhibits another class of piecewise analytical vector fields that satisfies Hypothesis~{\bf H}.  According to items (ii) and (iii) of Theorem \ref{the:01}, we have to consider the case when $p$ is a $\Sigma$-singular point of $X=(X^+,X^-)$ and satisfies $X^+(p)=0$ or $X^-(p)=0$.

\begin{proposition}\label{prop:09}
	Let $X=(X^+,X^-)$ be a piecewise analytical vector field and $\Sigma=f^{-1}(0)$ a curve, where $f:\mathbb R^2\rightarrow\mathbb R$ is analytic having $0$ as a regular value.  Assume that $p$ is a fold point of order $n^{\pm}$ of $X^{\pm}$ with $\pm(X^\pm)^{n^\pm}f(p)<0$ or $p$ is a singular point of $X^\pm$ having linear part with eigenvalues $\alpha\pm\beta i$, with $\beta\neq0$. Then,  if  $X^+f\cdot X^-f\mid_{V\setminus\{p\}}>0$ for some neighborhood $V$ of $p$, $X$ satisfies the Hypothesis~{\bf H}.
\end{proposition}

\begin{proof}
	We can assume that $p=(0,0)$ and $f(x,y)=y$.  If $p$ is a fold point of $X^+$ (resp. $X^-$) then we apply the ideas of proof of Proposition~\ref{lemaFold} using weight polar coordinates $(x,y)=(\rho \cos \theta, \rho^{n^+}\sin\theta)$ in $\Sigma\cup\Sigma^+$ (resp. $(x,y)=(\rho \cos \theta, \rho^{n^-}\sin\theta)$ in $\Sigma\cup\Sigma^-$). In this case, as we have seen in the proof of Proposition~\ref{lemaFold}, we have equation \eqref{eqpolarpeso} with $F^+ (\theta, 0)=0$  (resp. $F^- (\theta, 0)=0$) for all $\theta\in \mathbb R$ and $G^+(\theta, 0)\neq0$ for all $\theta\in[0,\pi]$ (resp. $G^-(\theta, 0)\neq0$ for all $\theta\in[-\pi,0]$).  If $p$ is a singular point of $X^+$ (resp. $X^-$) having linear part with eigenvalues $\alpha\pm\beta i$, with $\beta\neq0$, then we write the system as
		\begin{equation}\label{eq7}
		\begin{aligned}
			\dot{x} & =a x +b y + \sum_{i=2}^\infty P_i(x,y) , \\ 
			\dot{y} & =  c x + d y + \sum_{i=2}^\infty Q_i(x,y),
		\end{aligned}
	\end{equation}
	where $P_i$ and $Q_i$ are homogeneous polynomials of degree $i$ and $(a-d)^2+4bc<0$. Using polar coordinates $(x,y)=(\rho \cos \theta, \rho\sin\theta)$ system \eqref{eq7} becomes
	%\begin{equation}
	\[
	\begin{aligned}
		\dot{\rho} & = \alpha(\theta) \rho +\sum_{i=2}^\infty A_i(\theta)\rho^{i} , \\
		\dot{\theta} & =  \beta(\theta) +\sum_{i=2}^\infty R_i(\theta)\rho^{i-1} ,
	\end{aligned}
	\]
	%\end{equation}
	where $\alpha(\theta)=(a-d)\cos^2\theta+(b+c)\cos\theta\sin\theta+d$, $\beta(\theta)= c\cos^2\theta+(d-a)\cos\theta\sin\theta-b\sin^2\theta$,  $A_i(\theta)= P_i(\cos\theta,\sin\theta)\cos\theta+  Q_i(\cos\theta,\sin\theta)\sin\theta$ and $R_i(\theta)= Q_i(\cos\theta,\sin\theta)\cos\theta-  P_i(\cos\theta,\sin\theta)\sin\theta$. 
	We observe that 
	\[
	\beta(\theta) = c\left(\left(\cos\theta +\frac{(d-a)}{2c}\sin\theta\right)^2 - \frac{((a-d)^2+4bc)}{4c^2}\sin^2\theta \right)\!\!\cdot
	\]
	The fact that $(a-d)^2+4bc<0$ implies that $\beta(\theta)\neq0$ for all $\theta\in[0,2\pi]$. In this case is also clear that $X$ satisfies Hypothesis~{\bf H} because equation \eqref{eqpolarpeso} is given by
	\begin{equation*}\label{eq9}
		\frac{d\rho}{d\theta}  = \frac{\alpha(\theta) \rho +\sum_{i=2}^\infty A_i(\theta)\rho^{i}}{ \beta(\theta) +\sum_{i=2}^\infty R_i(\theta)\rho^{i-1}}\cdot
	\end{equation*} 
\end{proof}

\begin{lemma}
\label{lemaCusp}
	Let $Y:\mathbb R^2\rightarrow\mathbb R^2$ be an analytical vector field having $p$ as an isolated nilpotent singularity. If  $\Sigma$ is an analytical curve (i.e. $\Sigma=f^{-1}(0)$, where $f:\mathbb R^2\rightarrow\mathbb R$ is analytic having $0$ as a regular value) passing through $p$ transversal to the eigenspace of $DY(p)$, then there is an analytical change of coordinates such that in these new coordinates  $p=(0,0)$, $\Sigma$ is locally the set $\{(x,y):x=0\}$ and in a neighborhood of $p$ the analytical differential system associated to the vector field $Y$ writes as
	\begin{equation}\label{eq1}
		\begin{aligned}
			\dot{x} & = y, \\
			\dot{y} & = g(x)+yh(x)+y^2B(x,y),
		\end{aligned}
	\end{equation}
	where $g(x)=ax^m + o(x^m)$, for some $m\geq2$ with $a\neq0$, and the function $h$ is either identically null or there exist $b\neq0$ and $n\geq1$ such that $h(x)=bx^n + o(x^n)$.  %Moreover, the integer $m$ is even and either  $m<2n+1$ or $g(x)\equiv0$.
\end{lemma}

\begin{proof}
	It is easy to see that there are coordinates $(x, y)$ such that in these coordinates $p$ is the origin and locally $\Sigma=\{(x,y) : x=0\}$. As $p$ is a nilpotent singular point, the linear part of $Y$ at $p$ is a matrix
	\[
	DY(0, 0)=\left(\begin{array}{cc} a & b \\ c &d\end{array}\right)
	\] 
	satisfying $a+d=0$, $ad-bc=0$ and $a^2+b^2+c^2+d^2\neq0$. Now we prove that $b$ cannot be zero. If $b=0$ then $a=d=0$ and $c\neq0$. But in this case the eigenspace of $DY(0, 0)$ is parallel to $\Sigma$, which is not possible due to the hypothesis of transversality.
	
	The case $c=0$ implies $a=d=0$ and $b\neq0$.  In this case the linear change of coordinates $x=b\,u$ and $y=v$ preserves $\Sigma$ and transforms $DY(0, 0)$ into 
	\begin{equation}\label{eq2}
		\left(\begin{array}{cc} 0 & 1 \\ 0 &0\end{array}\right).
	\end{equation}
	
        	In the case $c\neq0$ we have $d=-a$ and $b=-a^2/c$. We use the linear change of coordinates $x=u$ and $y= c u/a -cv/a^2$ that preserves $\Sigma$ and transforms $DY(0, 0)$ into the same matrix presented in \eqref{eq2}.
	
	The differential system associated to the vector field $Y$ writes as
	\begin{equation}\label{n1}
		\begin{aligned}
			\dot{u} & = v + P_2(u,v), \\
			\dot{v} & = Q_2(u,v).
		\end{aligned}
	\end{equation}
	Now we perform the analytical change of coordinates $x=u$ and $y=v+P_2(u,v)$ that also preserves $\Sigma$. So, system \eqref{n1} becomes \eqref{eq1},
%	\begin{equation}\label{n2}
%		\begin{aligned}
%			\dot{x} & = y, \\
%			\dot{y} & = f(x)+yg(x)+y^2B(x,y),
%		\end{aligned}
%	\end{equation}
	where $g(x)=Q_2(x,0)$ and $h(x)=\displaystyle\frac{\partial Q_2}{\partial y}(x,0)$. As $p$ is an isolated singularity, then $g(x)=ax^m + o(x^m)$, for some $m\geq2$ and $a\neq0$. The function $h$ is either identically null or there exist $b\neq0$ and $n\geq1$ such that $h(x)=bx^n + o(x^n)$.
\end{proof}

\begin{proposition}\label{propcusp}
	Let $X=(X^+,X^-)$ be a piecewise analytical vector field and $\Sigma=f^{-1}(0)$ a curve, where $f:\mathbb R^2\rightarrow\mathbb R$ is analytic having $0$ as a regular value.  Assume that $p\in \Sigma$ is either a  monodromic nilpotent singular point or a nilpotent cusp of $X^+$ (resp. $X^-$) such that $\Sigma$ is transversal to the eigenspace of $DX^+(p)$ (resp. $DX^-(p)$),  locally the characteristic orbits of the cusp $p$ are contained in $\Sigma^-$ (resp. $\Sigma^+$) in the cusp case, and $p$ is a fold point of order $n^{-}$ of $X^{-}$ with $(X^-)^{n^-}f(p)>0$ (resp. $n^{+}$ of $X^{+}$ with $(X^+)^{n^+}f(p)<0$) or $p$ is a singular point of $X^-$ (resp. $X^+$) having linear part with eigenvalues $\alpha\pm\beta i$, with $\beta\neq0$. Then,  if  $X^+f\cdot X^-f\mid_{V\setminus\{p\}}>0$ for some neighborhood $V$ of $p$, $X$ satisfies the Hypothesis~{\bf H}.
\end{proposition}
\begin{proof}
First, we assume that $p$ is a cusp of $X^+$, and the case that $p$ is a cusp of $X^-$ is similar. By Lemma \ref{lemaCusp}, there is an analytical change of variables such that locally %in the new variables $(u, v)$
we have that $p=(0,0)$, $\Sigma$  is the %$v$
$y$ axis and the associated differential system to $X^+$ has the form \eqref{eq1}. As $p$ is a cusp, according to \cite[Theorem 3.5, p. 116]{ArtDumLli2006},  for $g$ and $h$ given by  \eqref{eq1},  we have that $g(y)=ay^m + o(y^m)$, for some even $m\geq2$ with $a\neq0$, and the function $h$ is either identically null or there exist $b\neq0$ and  $n\geq 1$ such that $h(y)=by^n + o(y^n)$ and $m < 2n+1$. Now, after the change of variables $(x, y)\mapsto (-y, x)$, locally, we have that $\Sigma$ becomes the $x$ axis, i.e., $f(x,y)=y$ and the system associated to $X^+$ writes as
\begin{equation}\label{n2}
		\begin{aligned}
			\dot{x} & = g(-y)+xh(-y)+x^2B(-y, x), \\
		\dot{y} & = -x.
         	\end{aligned}
	\end{equation}
	Using  weight polar coordinates $(x,y)=(\rho^{m+1} \cos \theta, \rho^2\sin\theta)$ system \eqref{n2} becomes
 \begin{equation}
 \label{n3}
\begin{aligned}
\dot{\rho} & = \frac{\rho^m( a\cos\theta\sin^m\theta-\cos\theta\sin\theta ) +{\mathcal O}(\rho^{m+1})}{(m-1)\cos^2\theta+2} , \\ \\
\dot{\theta} & =  \frac{\rho^{m-1}(-(m+1)\cos^{2}\theta -2a\sin^{m+1}\theta )+{\mathcal O}(\rho^m)}{(m-1)\cos^2\theta+2},
\end{aligned}
\end{equation}
Eliminating the independent variable from system \eqref{n3}, we reduce it to the differential equation
\begin{equation}
\label{n4}
\frac{d\rho}{d\theta}  =  \frac{\cos\theta\sin\theta  -a\cos\theta\sin^m\theta }{(m+1)\cos^{2}\theta +2a\sin^{m+1}\theta }\rho +{\mathcal O}(\rho^2).
\end{equation}
 
The local expression for the characteristic orbits of the cusp singularity, is 
\begin{equation}
\label{n5}
y=-\sqrt[m+1]{\dfrac{(m+1)}{2a}x^2}+o(x^\frac{2}{m+1}).
\end{equation}
Therefore, the hypothesis implies that $\Sigma^-$ contains the characteristic orbits of the cusp at the origin. So, we have that $a>0$. 

Observe that, for differential equation \eqref{n4}, $G^+(\theta, 0)=(m+1)\cos^{2}\theta +2a\sin^{m+1}\theta$, where $G^+$ corresponds to the map given by \eqref{eqpolarpeso}. Thus, we have that $G^+(\theta, 0)\neq0$ for all $\theta\in[0,\pi]$. Note that the pair $(\cos\theta, \sin \theta)$ such that $(m+1)\cos^{2}\theta +2a\sin^{m+1}\theta=0$ is a parametrization of the cusp curve 
\[
\left(x, -\sqrt[m+1]{\dfrac{(m+1)}{2a}x^2}\right). 
\]
Finally, if $p$ is a fold point or a singular point with complex eigenvalues  of $X^-$ then we apply the ideas of the proofs of Propositions~\ref{lemaFold}, \ref{prop:09} using weight polar coordinates  $(x,y)=(\rho \cos \theta, \rho^{n^-}\sin\theta)$ or polar coordinates $(x,y)=(\rho \cos \theta, \rho\sin\theta)$  in $\Sigma\cup\Sigma^-$, respectively. In this case, as we have seen in the proofs of Propositions~\ref{lemaFold}, \ref{prop:09}, we have an equation of the form \eqref{eqpolarpeso} with   $G^-(\theta, 0)\neq0$ for all $\theta\in[-\pi,0]$.  Hence, from hypothesis $X^+f\cdot X^-f\mid_{V\setminus\{p\}}>0$ and doing rescaling of the time variable (if necessary), for some neighborhood $V$ of $p$, $X$ satisfies the Hypothesis~{\bf H}.
	
Now, we assume that $p$ is a monodromic nilpotent singular point of $X^+$, and the case that $p$ is a monodromic nilpotent singular point of $X^-$ is similar. As in the previous cases, by Lemma \ref{lemaCusp}, there is an analytical change of variables such that locally we have that $p=(0,0)$, $\Sigma$  is the $y$ axis, and the differential system associated to $X^+$ has the form \eqref{eq1}. As $p$ is a monodromic nilpotent singular point, according to \cite[Theorem 3.5, p. 116]{ArtDumLli2006},  for $g$ and $h$ given by  \eqref{eq1},  we have that $g(y)=ay^m + o(y^m)$, for some odd $m\geq2$ with $a<0$, and  $h$ is either identically null or there exist $n\geq 1$ and $b\neq0$ such that $h(y)=by^n + o(y^n)$ with $m<2n+1$ or $m=2n+1$  and $b^2+4a(n+1)<0$. Now, after the change of variables $(x, y)\mapsto (-y, x)$, locally, we have that $\Sigma$ becomes the $x$ axis, and the system associated to $X^+$ writes as \eqref{n2}. Using  weight polar coordinates $(x,y)=( \rho^{\frac{m+1}{2}} \cos \theta, \rho \sin\theta)$, we reduce system \eqref{n2}  to the differential equation
\begin{equation}
\label{n6}
\frac{d\rho}{d\theta}  =  \frac{\sin2\theta\big(1+a\sin^{m-1}\theta\big)}{(m+1)\cos^{2}\theta -2a\sin^{m+1}\theta }\rho +{\mathcal O}(\rho^2),
\end{equation}
if $m<2n+1$ or it is reduced to the differential equation
\begin{equation}
\label{n7}
\footnotesize
\frac{d\rho}{d\theta}  =  \frac{\sin2\theta\big((-1)^nb\cos\theta\sin^{n-1}\theta-a\sin^{2n}\theta-1\big)}{2\big(a(\sin^{n+1}\theta)^2-(-1)^nb\cos\theta\sin^{n+1}\theta-(n+1)\cos^{2}\theta\big)}\rho +{\mathcal O}(\rho^2),
\end{equation}
if $m=2n+1$. For differential equations \eqref{n6} and \eqref{n7}, we have that $G^+(\theta, 0)$, given by \eqref{eqpolarpeso}, are the maps 
\[
G^+(\theta, 0)=(m+1)\cos^{2}\theta -2a\sin^{m+1}\theta
\] 
and 
\[
G^+(\theta, 0)=a(\sin^{n+1}\theta)^2+(-1)^{n+1}b\cos\theta\sin^{n+1}\theta-(n+1)\cos^{2}\theta,
\] 
respectively. For the first expression, as $m$ is odd and $a<0,$ we have that $G^+(\theta, 0)\neq0$ for all $\theta\in[0,\pi]$. The second expression is not null for $\theta\in\{0, \pi\}$, and for $\theta\in (0, \pi)$, we can write 
\begin{equation}
\label{n8}
\footnotesize
G^+(\theta, 0)=\cos^2\theta\left(a\left(\frac{\sin^{n+1}\theta}{\cos\theta}\right)^2+(-1)^{n+1}b\left(\frac{\sin^{n+1}\theta}{\cos\theta}\right)-(n+1)\right)
\end{equation}
Therefore, as $b^2+4a(n+1)<0$, it follows by \eqref{n8} that $G^+(\theta, 0)\neq0$ for all $\theta\in[0,\pi]$. Hence, as in the previous cases, $X$ satisfies the Hypothesis~{\bf H}.
\end{proof}

\begin{theorem}\label{fund_lemma}
	If a piecewise analytical vector field $X=(X^+, X^-)$ has a $\Sigma$-monodromic singular point and satisfies the Hypothesis~{\bf H}, then there is a choice of coordinates in $\Sigma$ such that in these coordinates the displacement function is analytic.
\end{theorem}
\begin{proof}
	Let $\rho^+=\rho^+(\theta,\rho_0)$ be the solution of \eqref{eqpolarpeso} that satisfies the initial condition $\rho^+(0,\rho_0)=\rho_0$. Similarly, let $\rho^-=\rho^-(\theta,\rho_1)$ be the solution of \eqref{eqpolarpeso} that satisfies the initial condition $\rho^-(0,\rho_1)=\rho_1$. From the Hypothesis~{\bf H} the functions $G^\pm(\theta, 0)$ do not vanish at the appropriated domain, then both functions $\rho_0\mapsto\rho^+(\pi,\rho_0)$ and $\rho_1\mapsto\rho^-(-\pi,\rho_1)$ are analytic. In order to compute the displacement function we start with a point $(x_0,0)\in\Sigma$, with $x_0>0$. The point $(x_0,0)$ corresponds to the points $(\rho_0,0)$ and $(\rho_1,0)$ in the weight polar coordinates, where $\rho_0=x_0^{1/p}$ and $\rho_1=x_0^{1/q}$. The half-return maps for $X^+$ and $-X^-$ (we consider $-X^-$ because to define the displacement function, we need to reverse the flow of $X^-$) are
	\[
	\Pi^+(x_0)=-(\rho^+(\pi,x_0^{1/p}))^p\quad \mbox{ and } \quad\Pi^-(x_0)=-(\rho^-(-\pi,x_0^{1/q}))^q,
	\]
	respectively. 
	As $\rho^\pm(\pm \pi,\cdot)$ are analytic, doing $x_0=w^{pq}$, we have that
	\[
	\Psi^+(w)=\Pi^+(w^{pq})=-(\rho^+(\pi,w^{q}))^p 
	\]
	and
	\[
	\Psi^-(w)=\Pi^-(w^{pq})=-(\rho^-(-\pi,w^{p}))^q
	\]
	are also analytic. 
	So the displacement function $\Psi =\Psi^+-\Psi^-$  is analytic. This finishes the proof.
\end{proof}

We say that a $\Sigma$-monodromic singular point $p$ is {\it free of limit cycles} if there exists a neighborhood of $p$ without crossing limit cycles.  

\begin{corollary}
	\label{cor:PRMA}
	With the hypothesis of Theorem \ref{fund_lemma}, we have that $p$ is free of limit cycles.
\end{corollary}

\begin{observation}
\label{obnil}
We conclude, from Proposition \ref{prop:09}, Theorem \ref{fund_lemma} and Corollary \ref{cor:PRMA},  that if $p$ is a $\Sigma$-monodromic singular point of a piecewise analytical planar vector field $X=(X^+, X^-)$ such that it is a fold point of order $n^{\pm}$ of $X^{\pm}$ with $\pm(X^\pm)^{n^\pm}f(p)<0$ or $p$ is a singular point of $X^\pm$ having a linear part with eigenvalues $\alpha\pm\beta i$, with $\beta\neq0$, then its displacement function is analytic, and it is free of limit cycles. Now, by Proposition \ref{propcusp}, Theorem \ref{fund_lemma} and Corollary \ref{cor:PRMA}, if $p$ is nilpotent singular point of cusp or monodromic type for $X^+$ (resp. $X^-$),  such that the switching curve $\Sigma$ is transversal to the eigenspace associated to $p$, and $p$ is a fold point of order $n^{-}$ of $X^{-}$ with $-(X^-)^{n^-}f(p)<0$ ( resp. $n^{+}$ of $X^{+}$ with $(X^+)^{n^+}f(p)<0$) or $p$ is a singular point of $X^-$ (resp. $X^+$) having a linear part with eigenvalues $\alpha\pm\beta i$, with $\beta\neq0$, then its displacement function is analytic, and it is free of limit cycles. Note that if $p$ is a $\Sigma$-monodromic singular point and $p$ is a nilpotent singular point of cusp type for $X^+$ or $X^-$, then the hypothesis of transversality in Proposition \ref{propcusp} is automatically satisfied. However, for the monodromic nilpotent case, we could have the eigenspace associated with the eigenvalue zero non-transversal to $\Sigma$.   Proposition \ref{propcusp} does not cover this last case.
For the other cases where $p$ is a  $\Sigma$-monodromic singular point and it is a nilpotent singular point of a type different from the mentioned above or a linearly zero singular point of $X^+$ or $X^-$, we only obtained a positive answer to Dulac's problem in the case where $X$ satisfies Hypothesis~{\bf H}.
 
\end{observation}

From now on, we call the $\Sigma$-monodromic singular points that are nilpotent singular points of at least one of the analytical vector fields $X^\pm$ of {\it $\Sigma$-monodromic nilpotent singular points}. Our best result for the $\Sigma$-monodromic nilpotent singular points is Proposition \ref{propcusp}. However, if more restrictive hypotheses are assumed, we can find other types of $\Sigma$-monodromic nilpotent singular points, not contemplated in this proposition, which also satisfy Hypothesis~{\bf H}. For example, if we assume from the start that $\Sigma$ is the $x$ axis, $X^+$ and $X^-$ are both in nilpotent normal form \eqref{eq1}, and the origin is a  $\Sigma$-monodromic nilpotent singular point obtained by a nilpotent cusp and a monodromic nilpotent singular point. We can easily prove that in this case, the  Hypothesis~{\bf H} is satisfied following the steps from the proof of Proposition 12, using polar coordinates with suitable weights for each one of the normal forms. Following this idea, we get the next proposition that allows us to find other restrictive cases of $\Sigma$-monodromic nilpotent singular points, satisfying the Hypothesis~{\bf H}, that are nilpotent singular points having an elliptic sector and a hyperbolic sector (see the case {\bf Cusp-(HE-singular point)} from Section \ref{sec:App} for a concrete example).
	\begin{proposition}
		Let $X=(X^+,X^-)$ be a piecewise  analytical vector field with switching curve  $\Sigma=\{y=0\}$,  with $X^-$ or $X^+$ given by $Y:\mathbb R^2\rightarrow\mathbb R^2$ where
		\begin{equation}\label{HO2}
		Y:\left\{ \begin{array}{l} \dot{x}=y, \\ \dot{y}=g(x)+yh(x)+y^2B(x,y), \end{array}\right.
		\end{equation}
		 $g(x)=ax^m + o(x^m)$, $h(x)=bx^n + o(x^n)$ and $B$ analytical. Assume $a<0$, $n$ odd, $m=2n+1$ and $b^2+4a(n+1)\geq0$. If $b>0$ (resp. $b<0$) then $Y$ does not have characteristic orbit in $\Sigma^-\cup\Sigma$ (resp. $\Sigma^+\cup\Sigma$) at the origin. Moreover, if $b>0$ and $Y=X^-$ (resp. $b<0$ and $Y=X^+$) then the half-Poincar\'e return map $\Pi^-$ (resp. $\Pi^+$) is well defined and it is analytic. 
\end{proposition}

\begin{proof}
		According to \cite[Theorem 3.5, p. 116]{ArtDumLli2006} the origin is a nilpotent singular point of $Y$ having an elliptic sector and a hyperbolic sector. Using  weight polar coordinates $(x,y)=( \rho \cos \theta, \rho^{n+1} \sin\theta)$, we reduce system \eqref{HO2}  to the differential equation
		\begin{equation}
		\label{n26n}
		\frac{d\rho}{d\theta}  =  \frac{(n+1)\sin^2\theta-a\sin\theta\cos^{2n+1}\theta-b\sin^2 \theta\cos^{n}\theta}{(n+1)\sin^{2}\theta -a\cos^{2n+2}\theta -b\sin\theta\cos^{n+1}\theta}\rho +{\mathcal O}(\rho^2).
		\end{equation} 
		In the differential equation \eqref{n26n}, we observe that $G(\theta, 0)=(n+1)\sin^{2}\theta -a\cos^{2n+2}\theta -b\sin\theta\cos^{n+1}\theta$ is not null for $\theta\in\{0, \pi,2\pi\}$, and for $\theta\in (0, \pi) \cup (\pi,2\pi)$, we can write 
		\begin{equation}
		\label{n27n}
		\footnotesize
		G(\theta, 0)=-\sin^2\theta\left(a\left(\frac{\cos^{n+1}\theta}{\sin\theta}\right)^2+b\left(\frac{\cos^{n+1}\theta}{\sin\theta}\right)-(n+1)\right).
		\end{equation}
		Therefore, as $b^2+4a(n+1)\geq0$, it follows that equation $au^2+bu-(n+1)=0$ has real solutions $u_1$ and $u_2$ (equal or distinct), given by
		\[u_1=\frac{-b+\sqrt{b^2+4a(n+1)}}{2a}\quad \mbox{and} \quad u_2=\frac{-b-\sqrt{b^2+4a(n+1)}}{2a}.\] 
		The hypothesis $a<0$ implies $u_1u_2>0$. Observe that if $\theta_0\not\in\{0,\pi,2\pi\}$ and $G(\theta_0,0)=0$ then $\cos^{n+1}\theta_0 = u_i\sin\theta_0$ for $i=1$ or $i=2$. Due to the hypothesis $n$ odd we have that the case $u_1,u_2>0$ implies $0<\theta_0<\pi$, and  the case $u_1,u_2<0$ implies $\pi<\theta_0<2\pi$. The case $u_1,u_2>0$ is equivalent to $b>0$ and the case $u_1,u_2<0$ is equivalent to $b<0$.\\ $\,$ \vspace{-.3cm} \\
		\indent Let $y=cx^k + o(x^k)$, with $c\neq0$, the local expression of the characteristic orbits of $Y$, given by \eqref{HO2}, at the origin. Differentiating with respect to $t$ and using expressions of \eqref{HO2} we obtain
		\[ax^{2n+1}+bcx^{n+k}+\cdots = c^2kx^{2k-1}+\cdots\]
		If $k<n+1$ then $2k-1<2n+1$ and $2k-1<n+k$. It implies $c=0$. But it is impossible because we have $c\neq0$. For $k=n+1$ we obtain the equation
		\[a\left(\frac{1}{c}\right)^2+b\left(\frac{1}{c}\right)-(n+1)=0,\]
		that have solutions $c=1/u_1$ and $c=1/u_2$. This implies that the characteristic orbits are given by 
		\[y = \frac{1}{u_1}x^{n+1}+\mathcal O(x^{n+2}) \quad\mbox{and}\quad y = \frac{1}{u_2}x^{n+1}+\mathcal O(x^{n+2}).\]
		Again due to the hypothesis $n$ odd we have that the case $u_1,u_2>0$ (resp. $u_1,u_2<0$), equivalently $b>0$ (resp. $b<0$), implies the characteristic orbits  of $Y$ at the origin are in $\Sigma^+$ (resp. $\Sigma^+$). This completes the proof.
\end{proof}

\section{Applications}
\label{sec:App}

In this section, we study the topological structure in the neighborhood of a $\Sigma$-monodromic singular point of piecewise analytical vector fields.  More precisely, we exhibit $\Sigma$-monodromic singular points formed by singular points of $X^-$ and $X^+$ that have not yet been described in the literature. We will apply our results to analyse four cases: Cusp-Fold$_2$;  Fold$_2$-Fold$_4$;  Cusp-Degenerate; and  Elementary-Degenerate. 

\begin{proposition}
\label{methpolar}
Let $X=(X^+, X^-)$ be a piecewise analytical vector field having an isolated $\Sigma$-monodromic singular point and satisfying the Hypothesis~{\bf H}. Then the half-Poincar\'e return maps $\Pi^\pm$ are 

\begin{align}\label{piplus}
		\Pi^+(x_0)  = & -(u_1^+(\pi))^p x_0-p (u_1^+(\pi))^{p-1} u_2^+(\pi)x_0^{\frac{p+1}{p}}  \nonumber\\ 
		&-\bigg(\frac{1}{2}(p-1)p(u_1^+(\pi))^{p-2}(u_2^+(\pi))^2 \\ & +p(u_1^+(\pi))^{p-1}u_3^+(\pi) \bigg)x_0^{\frac{p+2}{p}}- \dots \nonumber 
	\end{align}
	and
	\begin{align}\label{piminus}
		\Pi^-(x_0)  =& -(u_1^-(-\pi))^q x_0-q (u_1^-(-\pi))^{q-1} u_2^-(-\pi)x_0^{\frac{q+1}{q}}\nonumber  \\ 
		& - \bigg(\frac{1}{2}(q-1)q(u_1^-(-\pi))^{q-2}(u_2^-(-\pi))^2 \\ & +q(u_1^-(-\pi))^{q-1}u_3^-(-\pi)\bigg)x_0^{\frac{q+2}{q}}- \cdots,\nonumber 
	\end{align}
	where $u_{i}^\pm(\theta)$ is the coefficient of order $i$ in the expansion of solution $\rho=\rho^\pm(\theta, \rho_0)$ (with $\rho^\pm(0, \rho_0)=\rho_0$) of equation \eqref{eqpolarpeso} in powers of $\rho_0$. Moreover, doing $x_0=w^{pq}$, we can obtain the expansion of the displacement function $\Psi$ in powers of $w$ given by $\Psi (w)=\Pi^+(w^{pq}) -\Pi^-(w^{pq})$.
\end{proposition}

\begin{proof}
	We will use the method described in \cite[p.~249]{ALGM} to get Taylor series expansion at the origin of the displacement function and so to study the Center-Focus Problem for $\Sigma$-monodromic singular points. More precisely, the right hand side of equation \eqref{eqpolarpeso} is an analytical map in the squares  $[0, \pi]\times (-\epsilon, \epsilon)$ and $[-\pi,0]\times (-\epsilon, \epsilon)$, for some $\epsilon>0$. Then we can write equation \eqref{eqpolarpeso} as 
	\begin{equation}\label{polarpesoexpan}
		\frac{d\rho}{d\theta}  = R_1^\pm(\theta) \rho+R_2^\pm(\theta) \rho^2+R_3^\pm(\theta) \rho^3+\cdots.
	\end{equation}

	Let $\rho=\rho^+(\theta,\rho_0)$ the solution of equation \eqref{polarpesoexpan}$^+$ with the initial condition $\rho^+(0, \rho_0)=\rho_0$. 
	The Taylor series of the solution $\rho^+$ on powers of $\rho_0$ in a	neighborhood of $0$ has the form 
	\begin{equation}\label{polarsolexpan}
		\rho^+(\theta,\rho_0)= u_1^+(\theta) \rho_0+u_2^+(\theta) \rho_0^2+u_3^+(\theta) \rho_0^3+\cdots,
	\end{equation}
	with $u_1^+(0)=1$ and $u_2^+(0)=u_3^+(0)=\cdots=0$.
	
	\smallskip
	
	Substituting \eqref{polarsolexpan} into \eqref{polarpesoexpan}$^+$, and comparing the coefficients of the corresponding powers of $\rho_0$ of the right and left hand sides, we obtain the following recursive differential equations for the coefficients $u_i^+(\theta)$ ($i=1, 2, 3, \cdots$)
	\begin{align}\label{recurpolar}
		(u_1^+)' (\theta) = & R_1^+(\theta)u_1^+(\theta),\nonumber\\
		(u_2^+)' (\theta) =& R_1^+(\theta)u_2^+(\theta)+R_2^+(\theta)(u_1^+(\theta))^2,\\
		(u_3^+)' (\theta)= & R_1^+(\theta)u_3^+(\theta)+2R_2^+(\theta)u_1^+(\theta)u_2^+(\theta)+R_3^+(\theta)(u_1^+(\theta))^3,\nonumber\\
		\vdots \hspace{0.1cm}& \nonumber
	\end{align}
	
	Using these initial conditions $u_1^+(0)=1$ and $u_2^+(0)=u_3^+(0)=\cdots=0$ and successively integrating equations \eqref{recurpolar} as linear differential equations, we obtain
	\begin{equation}\label{solpolarpeso}
		\begin{array}{lcl}
			u_1^+(\theta) & = & e^{\int_0^\theta R_1^+(\theta)d\theta},\\
			u_2^+(\theta)  & = & e^{\int_0^\theta R_1^+(\theta)d\theta}\int_0^\theta R_2^+(\theta)u_1^+(\theta) d\theta,\\
			u_3^+ (\theta) & = & e^{\int_0^\theta R_1^+(\theta)d\theta}\int_0^\theta\left( 2R_2^+(\theta)u_2^+(\theta)+R_3^+(\theta)(u_1^+(\theta))^2 \right)d\theta,\\
			& \vdots & 
		\end{array}
	\end{equation}
	Now, following the proof of Lemma \ref{fund_lemma}, the half-Poincar\'e return map for $X^+$ is $\Pi^+(x_0)  = -(\rho^+(\pi,x_0^{1/p}))^p$, i.e., $\Pi^+$ is given by \eqref{piplus}.

	Analogously, we obtain that the half-Poincar\'e return map $\Pi^-(x_0)$  for $-X^-$ is given  by \eqref{piminus}. 
	\end{proof}
	
	In what follows we consider piecewise analytical vector fields  $X=(X^+,X^-)$ having the origin as a $\Sigma$-monodromic singular point and the switching curve  $\Sigma=\{y=0\}$. Now, we will use Proposition~\ref{methpolar} considering some classes of piecewise analytical vector fields.  The figures that illustrate the cases below were made with the help of the software P5  \cite{P5}. %(Piecewise Polynomial Planar Phase Portraits-P5, from authors Chris Herssens,  Peter de Maesschalck,  Joan C. Art\'es, Freddy Dumortier and Jaume Llibre, see {\it https://www.uhasselt.be/UH/dysy/Software/P5.html} for more details).
\medskip

%\begin{application}[\textbf{Cusp-fold case}]\label{ex1}
\noindent\textbf{Cusp-Fold$_2$ case:}
%In this example the vector field $X^+$ has a cusp singularity at the origin.  
The classical normal form of cusp singular point is given by $(\dot{x},\dot{y})=(y,ax^m+bx^ny)$, with $m$ even and $m<2n+1$. To simplificate the calculations, we will consider the case $m=2,\ n=1$ and $a=1$. In order to put the characteristic orbits of the cusp singular point tangent to the $y$-axis we perform the change of variable $\overline{x}=y$,  $\overline{y}=-x$ and $\overline{t}=-t$. We obtain 
\begin{equation}\label{cusp}
X^+:\left\{ \begin{array}{l} \dot{x}=-y^2+bxy, \\ \dot{y}=x. \end{array}\right.
\end{equation}
In the half-plane $\Sigma^-$ we consider the vector field $X^-$ with a fold at origin
\begin{equation}\label{fold}X^-:\left\{ \begin{array}{l} \dot{x}=1, \\ \dot{y}=x. \end{array}\right.
\end{equation}
Note that $X^-$ has a fold point of order $2$ at the origin, because $X^-(0,0)\neq (0,0)$,  $X^-f(0, 0)=0$ and $(X^-)^2f(0,0)=1$. Therefore, by Theorem \ref{the:01}~(ii), the origin is a $\Sigma$-monodromic singular point. Note that the trajectories of $X$ are oriented counterclockwise sense. See Figure~\ref{cusp_fold2}.
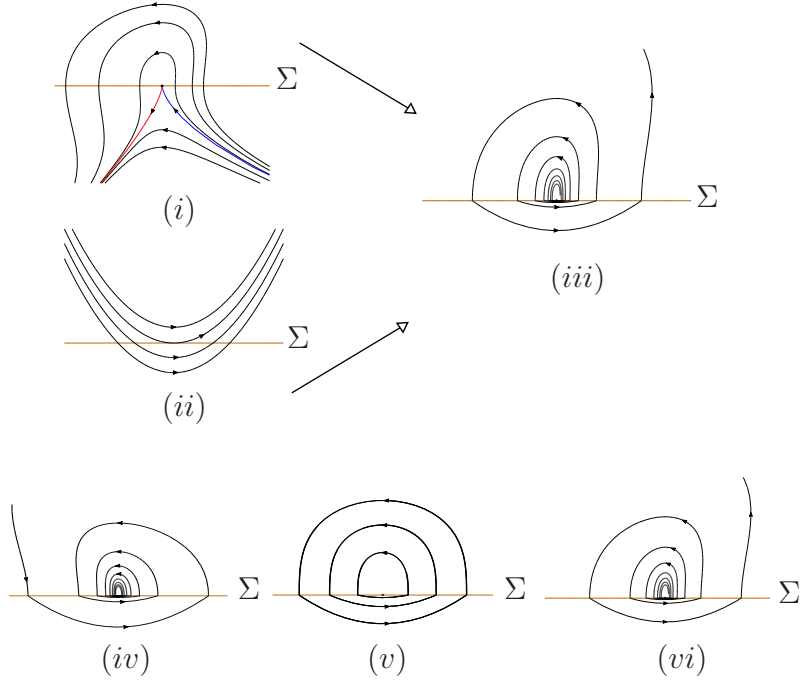
\begin{figure}[h]
	\begin{center}		
		\begin{overpic}[width=0.80\textwidth]{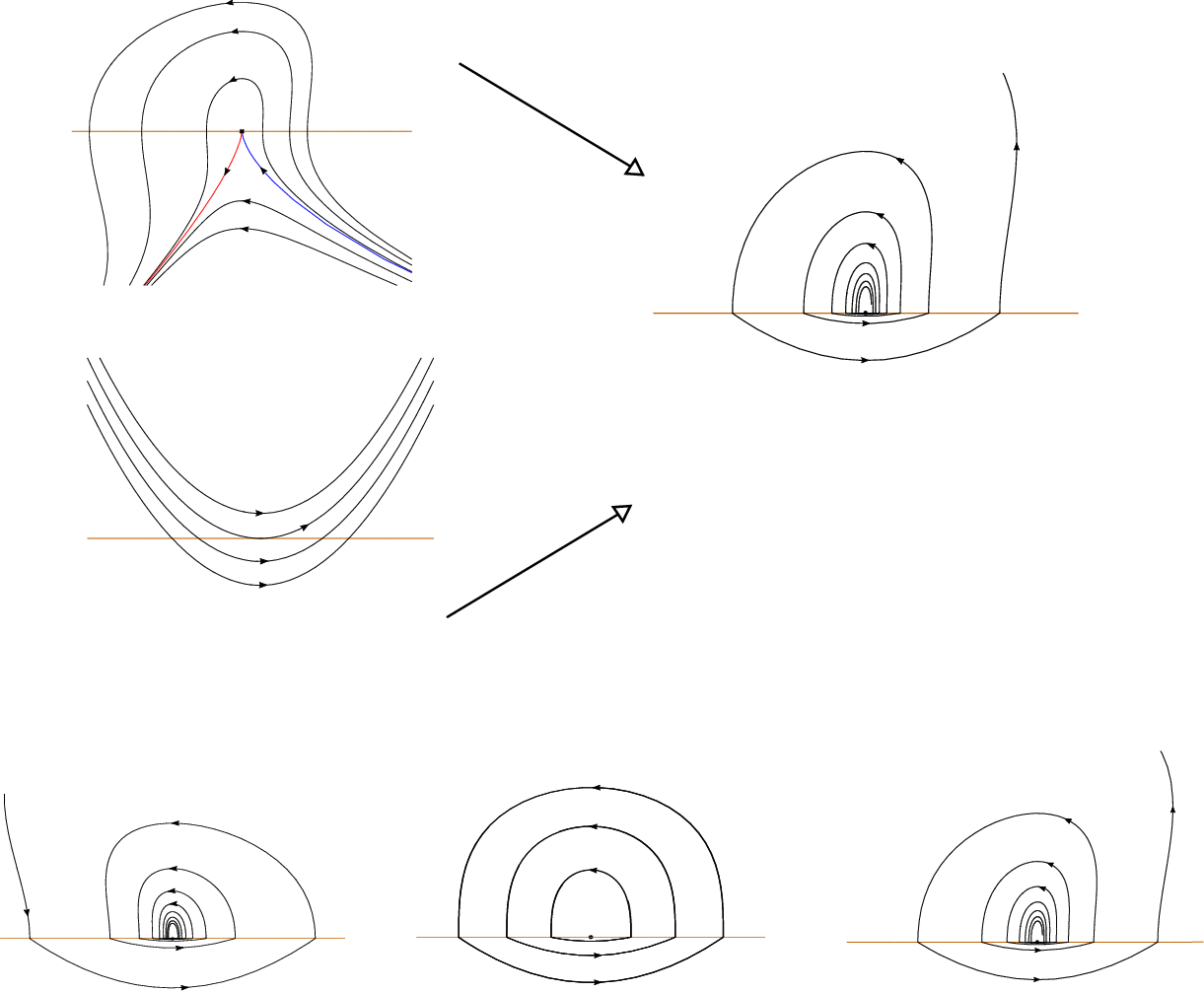}
			  	%\begin{overpic}[grid,width=0.95\textwidth]{CompletoCuspFold.eps}
			\put(20,54) {$(i)$}
			\put(20,28) {$(ii)$}
			\put(71,45){$(iii)$}
			\put(12,-5) {$(iv)$}
			\put(47,-5) {$(v)$}
			\put(85,-5){$(vi)$}
			\put(35,71){$\Sigma$}
			\put(36.5,37){$\Sigma$}
			\put(90.2,55.5){$\Sigma$}
			\put(30,4){$\Sigma$}
			\put(65,4){$\Sigma$}
			\put(101,3.5){$\Sigma$}
			%\put(85,26){$x_0$}
		\end{overpic}
	\end{center}
	\vspace{0.7cm}
	\caption{Local phase portraits at the origin. $(i)$ A cusp for  $X^+$ with $b>0$.  $(ii)$ A fold$_2$ for $X^-$. $(iii)$ The Cusp-Fold$_2$ case for $X$ with $b>0$. $(iv)$ A stable $\Sigma$-focus for $X$ with $b<0$. $(v)$ A $\Sigma$-center for $X$ with $b=0$. $(vi)$ An unstable  $\Sigma$-focus for $X$ with $b>0$.}
	\label{cusp_fold2}
\end{figure}

We must compute the half-Poincar\'e return map $\Pi^-:V\subset \Sigma \rightarrow W\subset \Sigma$ associated with vector field $-X^-$ at $(x_0,0)\in V$, with $x_0>0$. The solution $\alpha=\alpha(t)$ of \eqref{fold} that satisfy the initial condition $\alpha(0)=(x_0,0)$ is $\alpha(t)=(t+x_0,t^2/2+x_0t)$. The values of $t$ such that the second component of $\alpha(t)$ is equal to zero is $t=0$ or $t=-2x_0$. So, $\Pi^-(x_0)=-x_0$, because  $\alpha(-2x_0)=(-x_0,0)$. 

The next step is to compute the half-Poincar\'e return map $\Pi^+:V\rightarrow W$. Following the proof of Proposition~\ref{methpolar} we need to determine the weights $p$ and $q$. For this, we use the Newton polygon (see  \cite[p.~104]{ArtDumLli2006}) associated with system \eqref{cusp}. Here, the Newton polygon is defined by the vertices  $(0,1),\ (-1,2),\ (1,-1)$ and the straight line connecting  $(-1,2)$ and $(1,-1)$  is $3x+2y=1$. So, we perform the change of coordinates $x=\rho^3\cos \theta$ and $y=\rho^2\sin \theta$. In these coordinates $(\rho,\theta)$ system \eqref{cusp} writes
\begin{equation}\label{cuspp}
X^+:\left\{ \begin{array}{l} \dot{\rho}=f_2^+(\theta)\rho^2+f_3^+(\theta)\rho^3, \\ \dot{\theta}=g_1^+(\theta)\rho + g_2^+(\theta)\rho^2. \end{array}\right.
\end{equation}
where 
\begin{align*}
f_2^+(\theta) & =\frac{\cos \theta\sin \theta(1-\sin \theta)}{3\cos^2\theta+2\sin^2\theta}, &
f _3^+(\theta)  =   \frac{b\cos^2 \theta\sin \theta}{3\cos^2\theta+2\sin^2\theta}, \\
g_1^+(\theta) & = \frac{2\sin^3 \theta+3\cos^2 \theta}{3\cos^2\theta+2\sin^2\theta}, &
g_2^+(\theta)  =   \frac{-2b\cos \theta\sin^2 \theta}{3\cos^2\theta+2\sin^2\theta}.
\end{align*}

We observe that $X$ satisfies the Hypothesis~{\bf H}, because $g_1^+(\theta)\neq0$ for all $\theta\in [0,\pi]$. Considering $\theta$ as new independent variable, system \eqref{cuspp} is equivalent to equation
\begin{equation}\label{cusp3}
\frac{d\rho}{d\theta}=R_1^+(\theta)\rho+R^+_2(\theta)\rho^2 + {\mathcal O}(\rho^3).
\end{equation}
where
\[
R^+_1(\theta)=\frac{f_2^+(\theta)}{g_1^+(\theta)} \mbox{ and } R_2^+(\theta)=\frac{f_3^+(\theta)g_1^+(\theta)-f_2^+(\theta)g_2^+(\theta)}{g_1^+(\theta)^2}.
\]

According to the proof of Proposition \ref{methpolar}, the analytical solution $\rho=\rho(\theta,\rho_0)$ of the equation \eqref{cusp3} satisfying the initial condition $\rho(0,\rho_0)=\rho_0$ is given by \eqref{polarsolexpan}. By \eqref{solpolarpeso} it follows that
\begin{align*}
u_1(\theta) & =\frac{3^{1/6}}{m(\theta)^{1/6}},  \\
u_2(\theta) & = \frac{b}{m(\theta)^{1/6}}\int_0^\theta \frac{3^{1/6} (11\sin(3x)+10\sin(x)+\sin(5x))}{16\,m(x)^{13/6}} dx,
\end{align*}
with $m(\theta)=(3-2\sin \theta)\cos^2 \theta+2\sin \theta$. From these computations we have $u_1(\pi)=1$ and $u_2(\pi)=bk_0$ with $k_0>0$. So we have  
\[
\rho(\pi,\rho_0)=\rho_0+bk_0\rho_0^2 + {\mathcal O}(\rho_0^3).
\] 

Hence, by \eqref{piplus}, the half-Poincar\'e return map with respect to $\Sigma^+$ is
\[
\Pi^+(x_0)=-x_0-3bk_0x_0^{4/3}+{\mathcal O}(x_0^{5/3}),
\]
and the displacement function is
\[
\Psi(x_0)=\Pi^+(x_0)-\Pi^-(x_0)=-3bk_0x_0^{4/3}+{\mathcal O}(x_0^{5/3}).
\]

For $x_0$ small enough and $b<0$ we have $\Psi(x_0)>0$. It implies $\Pi^+(x_0) >\Pi^-(x_0)$ and the origin is a stable $\Sigma$-focus. If $b>0$ then the origin is an unstable $\Sigma$-focus. To finish this analysis, remains to show that for $b=0$ we have a $\Sigma$-center. This is easy because system \eqref{cusp} with $b=0$ is Hamiltonian with energy $H(x,y)= x^2/2+y^3/3$. The points $(x_0,0)$ and $(-x_0,0)$ are in the same level set of $H$. Clearly it implies that $\Pi^+(x_0)=-x_0$. So, the origin has a neighborhood free of limit cycles.
%\end{application}

\medskip

%\begin{application}[\textbf{Cusp-degenerate singular point}]\label{ex3}
\noindent\textbf{Cusp-(HE-singular point):}
Consider a $\Sigma$-monodromic singular point formed by a cusp singular point for $X^+$ and a degenerate singular point with an elliptic sector and a hyperbolic sector for $X^-$. On the half-plane $\Sigma^+$,  $X^+$ is associated with system \eqref{cusp} 
and in the half-plane $\Sigma^-$ we consider
\begin{equation}\label{HE}X^-:\left\{ \begin{array}{l} \dot{x}=-y, \\ \dot{y}=x^3-4xy. \end{array}\right.
\end{equation}

The vector field~\eqref{cusp} that has at origin a cusp singular point of $X^+$ and the characteristic orbits of the cusp are tangent to $y$ axis and belongs to half-plane $\Sigma^-$. 

For $X^-$ the origin as a nilpotent singular point with one elliptic sector and one hyperbolic sector, called {\it HE-singular point}. See Figure~\ref{cusp-degen} 
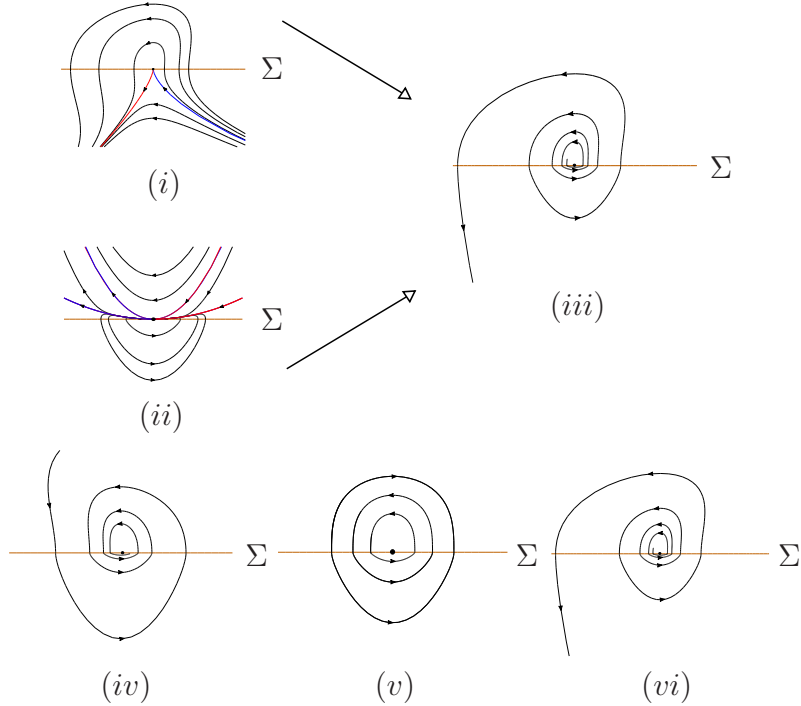
\begin{figure}[h]
	\begin{center}		
		\begin{overpic}[width=0.80\textwidth]{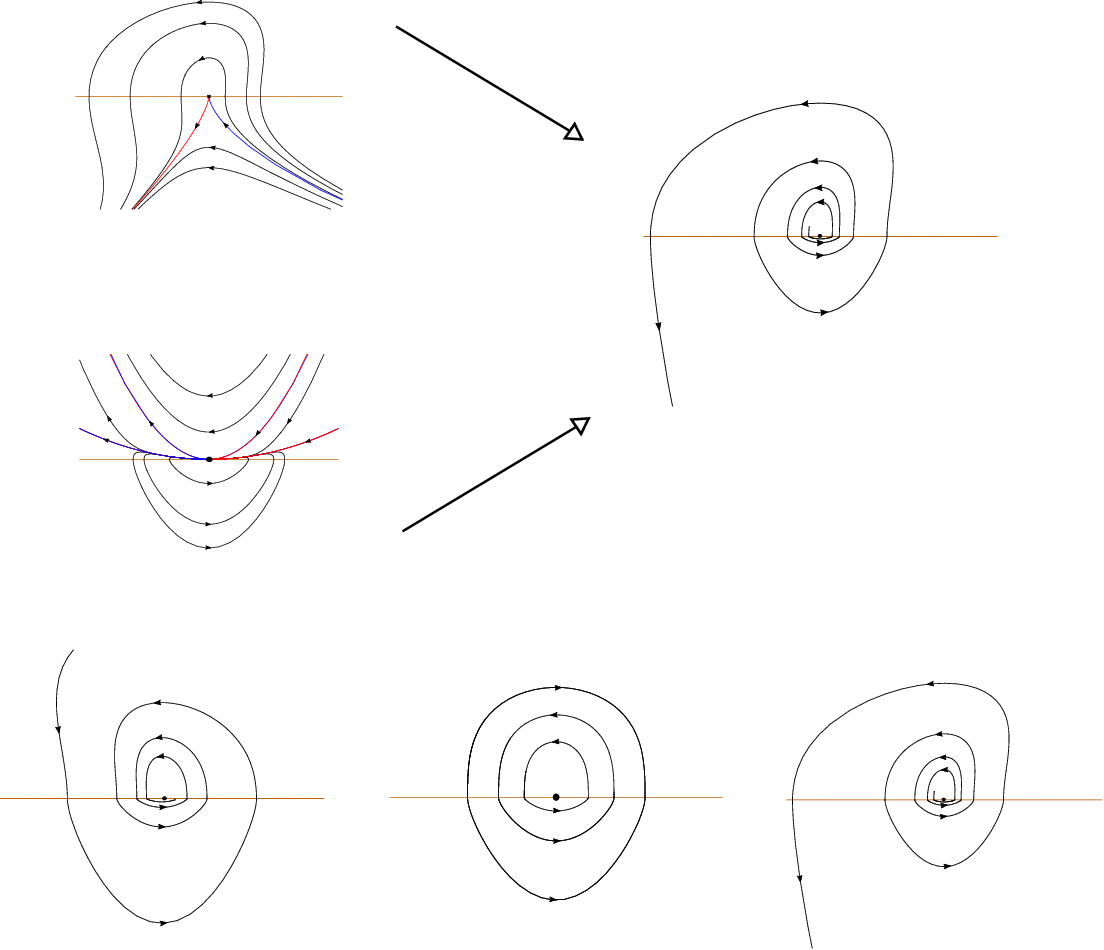}
			  	%\begin{overpic}[grid,width=0.95\textwidth]{CompletoCupHE.eps}
			\put(18,61) {$(i)$}
			\put(17,30.5) {$(ii)$}
			\put(71,45){$(iii)$}
			\put(12,-5) {$(iv)$}
			\put(48,-5) {$(v)$}
			\put(83,-5){$(vi)$}
			\put(33,76){$\Sigma$}
			\put(33,43){$\Sigma$}
			\put(92,63){$\Sigma$}
			\put(31,12){$\Sigma$}
			\put(66.5,12){$\Sigma$}
			\put(101,12){$\Sigma$}
			%\put(85,26){$x_0$}
		\end{overpic}
	\end{center}
	\vspace{0.7cm}
	\caption{Local phase portraits at the origin. $(i)$ A cusp for $X^+$,with $b>0$.   $(ii)$ A HE-singular point for $X^-$.  $(iii)$  The Cusp-(HE-singular point) case for $X$ with $b>0$. $(iv)$ A stable $\Sigma$-focus for $X$ with $b<0$.  $(v)$ A $\Sigma$-center for $X$ with $b=0$. $(vi)$ An unstable $\Sigma$-focus for $X$ with $b>0$.}
	\label{cusp-degen}
\end{figure}

 The parabolas  $\mathcal{P_\pm}: y=(2\pm\sqrt{2})^{-1}x^2$ are invariant by the flow of $X^-$, because $\langle X^-, \nabla h_\pm\rangle=h_\pm k$ with $h_\pm(x,y)=x^2+(-2\pm \sqrt{2})y$ and $k_\pm(x,y)=(-2\pm\sqrt{2})x$ (i.e. $h_\pm=0$ is a {\it algebraic invariant curve} of $X^-$, see  \cite[p.~215]{ArtDumLli2006}). Moreover $X^-$ is transversal to $x$ axis except in the origin. Therefore the separatrices of the hyperbolic sector of  the HE-singular point of $X^-$ belongs to half-plane $\Sigma^+$.

It follows from Theorem \ref{the:01}, that the origin is a $\Sigma$-monodromic singular point of $X$. Note that the trajectories of $X$ are oriented counterclockwise sense. As system \eqref{HE} is invariant by the change of variables $(x, y, t)\mapsto (-x, y, -t)$, it is a reversible system and so  $\Pi^-(x_0)=-x_0$, for $x_0>0$. So as in the Cusp-Fold case, the origin is a stable $\Sigma$-focus if $b<0$,  an unstable  $\Sigma$-focus if $b>0$,  and a $\Sigma$-center if $b=0$. This analysis proves that the origin has a neighborhood free of limit cycles.
%\end{application}

\smallskip

\noindent\textbf{Fold$_2$-Fold$_4$:}
Consider a $\Sigma$-monodromic singular point where both $X^+$ and $X^-$ have fold points at the origin. In the half-plane $\Sigma^+$ we have
\begin{equation}\label{2fold}
 X^+:\left\{ \begin{array}{l} \dot{x}=a x-1, \\ \dot{y}=x+by, \end{array}\right.
\end{equation}
and in the half-plane $\Sigma^-$ we have
\begin{equation}\label{4fold}X^-:\left\{ \begin{array}{l} \dot{x}=y+1, \\ \dot{y}=x^3+cx^2y. \end{array}\right.
\end{equation}

As $X^+(0, 0)\neq (0, 0)$, $X^+f(0, 0)=0$ and $(X^+)^2f(0,0)=-1$, the origin is a fold point of order $2$ of $X^+$, called fold$_2$. The origin is a fold point of order $4$ of $X^-$ (called fold$_4$), provided that $X^-(0, 0) \neq (0, 0)$, $X^-f(0,0)=(X^-)^2f(0,0)=(X^+)^3f(0,0)=0$ and $(X^+)^4f(0,0)=6$. Therefore, by Theorem \ref{the:01}, the origin is a $\Sigma$-monodromic singular point. Now, by Proposition \ref{lemaFold} and Theorem \ref{fund_lemma}, the Poincar\'e return map is analytic. Note that the trajectories of $X$ are oriented counterclockwise. See Figure~\ref{fold2-fold4}.
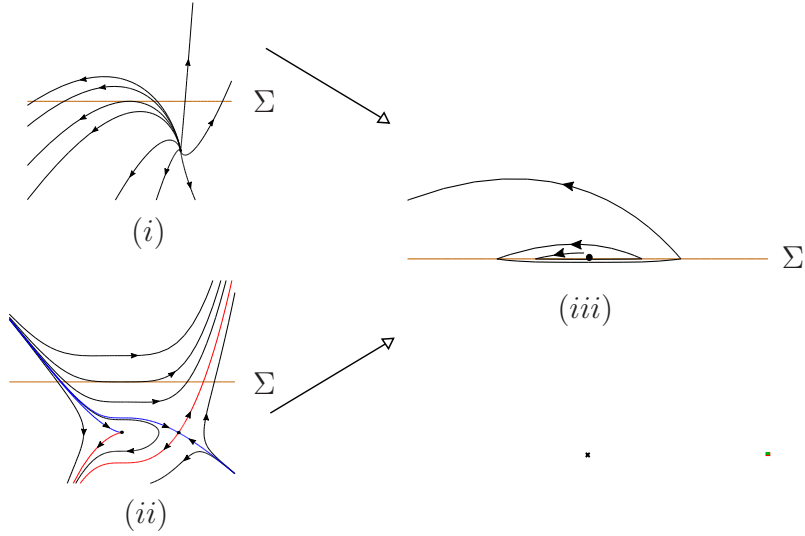
\begin{figure}[h]
	\begin{center}		
		\begin{overpic}[width=0.80\textwidth]{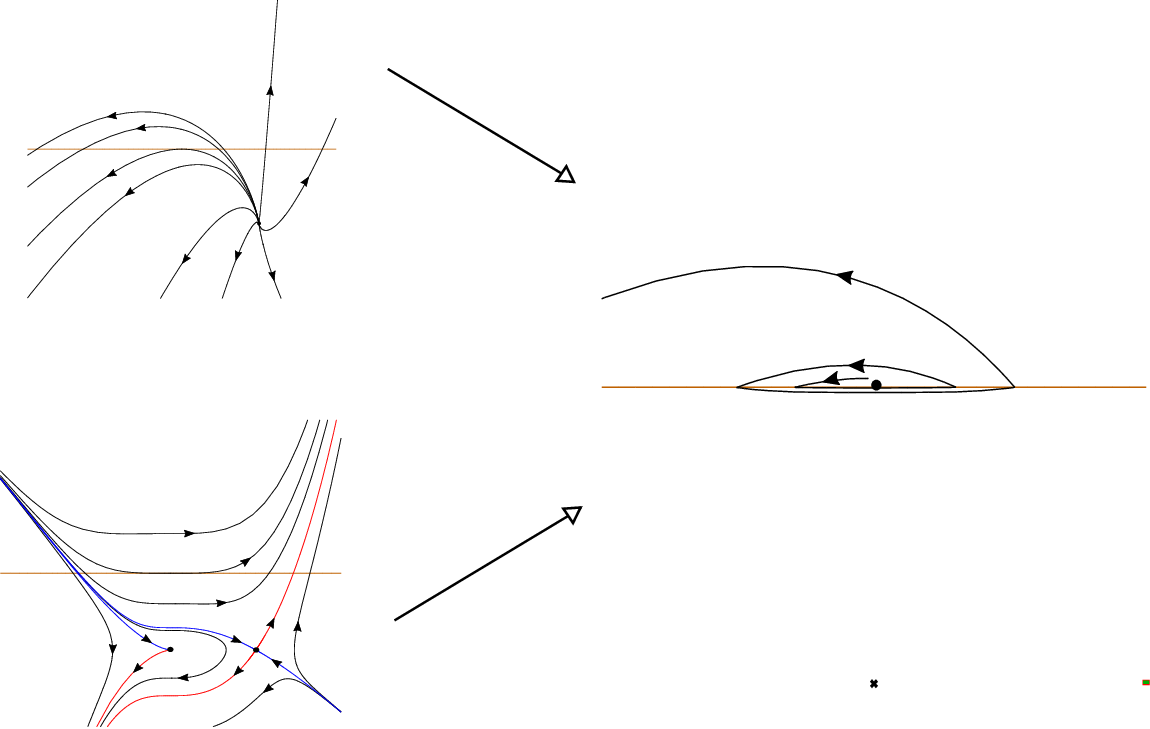}
			  	%\begin{overpic}[grid,width=0.95\textwidth]{DescontFold2Fold4.eps}
			\put(16,32) {$(i)$}
			\put(15,-5) {$(ii)$}
			\put(72,22){$(iii)$}
			%\put(24,10) {$0$}
			%\put(24,37) {$0$}
			%\put(77,25){$0$}
			\put(32,49){$\Sigma$}
			\put(32,12){$\Sigma$}
			\put(101.5,28.5){$\Sigma$}
			%\put(83,20){$x_0$}
		\end{overpic}
	\end{center}
	\vspace{0.7cm}
	\caption{Local phase portraits at the origin. $(i)$ A fold$_2$ for $X^+$ with $a=b=1$.   $(ii)$  A fold$_4$ for $X^-$ with $c=1$.  $(iii)$ The  
	Fold$_2$-Fold$_4$ case for $X$, with $a=b=c=1$.}
	\label{fold2-fold4}
\end{figure}  

Doing the change of variables $(x,y)=(\rho\cos\theta,\rho^2\sin\theta)$ in $\Sigma\cup\Sigma^+$ and $(x,y)=(\rho\cos\theta,\rho^4\sin\theta)$ in $\Sigma\cup\Sigma^-$ the differential systems associated with vector fields $X^+$ and $X^-$ are equivalent to the differential equations
\begin{equation}
\label{eqpolar2fold}
\begin{array}{lcl}
\dfrac{d\rho}{d\theta} &=&\dfrac{f_1^+(\theta) \rho+f_2^+(\theta)\rho^2}{g_0^+ (\theta)+g_1^+ (\theta) \rho} \\ \\
&=&\dfrac{f_1^+(\theta)}{g_0^+(\theta)}\rho+\dfrac{f_2^+(\theta)g_0^+(\theta)-f_1^+(\theta)g_1^+(\theta)}{g_0^+(\theta)^2}\rho^2\\  \\& & +\dfrac{g_1^+(\theta)(f_1^+(\theta)g_1^+(\theta)-f_2^+(\theta)g_0^+(\theta))}{g_0^+(\theta)^3}\rho^3+{\mathcal O}(\rho^{4})
\end{array}
\end{equation}
and
\begin{equation}
\label{eqpolar4fold}
\begin{array}{lcl}
\dfrac{d\rho}{d\theta} &=&\dfrac{f_1^-(\theta) \rho+f_4^-(\theta)\rho^4+f_5^-(\theta)\rho^5}{g_0^- (\theta)+g_3^- (\theta) \rho^3+g_4^- (\theta) \rho^4} \\ \\
&=&\dfrac{f_1^-(\theta)}{g_0^-(\theta)}\rho+\dfrac{f_4^-(\theta)g_0^-(\theta)-f_1^-(\theta)g_3^-(\theta)}{g_0^-(\theta)^2}\rho^4+{\mathcal O}(\rho^{5}), 
\end{array}
\end{equation}
respectively, where $f_1^+(\theta)=\cos\theta(\sin\theta -1)$, $f_2^+(\theta)=(a-b)\cos^2\theta+b$, $g_0^+(\theta)=\cos^2\theta+2\sin\theta$, $g_1^+(\theta)=\cos\theta\sin\theta(b -2a)$, $f_1^-(\theta)=-\cos\theta(1+\sin\theta\cos^2\theta)$, $f_4^-(\theta)=-c\cos^2\theta\sin^2\theta$, $f_5^-(\theta)=-\cos\theta\sin\theta$, $g_0^-(\theta)=\cos^4\theta-4\sin\theta$, $g_3^-(\theta)=c\sin\theta\cos^3\theta$ and $g_4^-(\theta)=-4\sin^2\theta$.

Let  $\rho^\pm(\theta,\rho_0)$ be the analytical solution  of equations \eqref{eqpolar2fold} and \eqref{eqpolar4fold} satisfying the initial condition $\rho^\pm(0,\rho_0)=\rho_0$ with $\rho_0> 0$. Using the method described in the proof of Proposition \ref{methpolar}, we can obtain the series expansions of $\rho^\pm(\theta,\rho_0)$ in powers of $\rho_0$. Now, as $p=q=1$ in \eqref{piplus} and \eqref{piminus}, it follows that $x_0=\rho_0$ and  the displacement function is
\begin{align*}
\Psi(\rho_0)  =& \rho^+(\pi, \rho_0)-\rho^-(-\pi, \rho_0)= \dfrac{2(a+b)}{3}\rho_0^2+\left(\dfrac{4(a+b)^2}{9}\right)\rho_0^3\\
& +\left(u_4^+(\pi)-c\int_0^{-\pi}\frac{\sin\theta\cos^2\theta(3\cos2\theta-4)}{(\cos^4\theta-4\sin\theta)^{\frac{11}{4}}}d\theta\right)\rho_0^4+ {\mathcal O}(\rho_0^{5}).
\end{align*}

{For} $a+b\neq 0$, we were not able to compute the expression of $u_4^+(\pi)$, but it is not necessary to study the stability of the origin in this case. Now, for $a+b=0$, we have  $u_4^+(\pi)=0$. Hence, for $\rho_0$ small enough and $a+b<0$ we have $\Psi(\rho_0)<0$, i.e. $\rho^+(\pi, \rho_0) < \rho^-(-\pi, \rho_0)$, and so the origin is a stable $\Sigma$-focus. On the other hand, if $a+b>0$ then the origin is an unstable  $\Sigma$-focus. If $a+b=0$ and $c>0$, as 
\[\int_0^{-\pi}\frac{\sin\theta\cos^2\theta(3\cos2\theta-4)}{(\cos^4\theta-4\sin\theta)^{\frac{11}{4}}}d\theta> 0,\]
we get that the origin is a stable $\Sigma$-focus and if $c<0$ it is an unstable  $\Sigma$-focus. In order to finish the analysis of this example remains to show that for $b=-a$ and $c=0$ we have a $\Sigma$-center. This follows from the differential systems \eqref{2fold} with $b=-a$ and \eqref{4fold} with $c=0$, because they are Hamiltonian with energies $H^+(x,y)= x^2/2-axy+y$ and $H^-(x,y)= x^4/4-y^2/2-y$, respectively. Note that the points $(x_0,0)$ and $(-x_0,0)$ are in the same level set of $H^\pm$. Clearly it implies that $\Psi(\rho_0)\equiv 0$.

\medskip

\noindent\textbf{Elementary-Degenerate:}
Consider a $\Sigma$-monodromic singular point where the origin is an elementary singular point for $X^+$ and it is a degenerate one for $X^-$. More precisely, on the half-plane $\Sigma^+$,  $X^+$ is associated with the differential system
\begin{equation}\label{E}
X^+:\left\{ \begin{array}{l} \dot{x}=a x-by+c x^2, \\ \dot{y}=b x+ay, \end{array}\right.
\end{equation}
with $b>0$, and on the half-plane $\Sigma^-$ we have
\begin{equation}\label{D}
X^-:\left\{ \begin{array}{l} \dot{x}=-y^3, \\ \dot{y}=x^3+d x^4 y. \end{array}\right.
\end{equation}

Note that the origin is a monodromic singular point of $X^+$. It follows from  Propositions 5 and 10 of \cite{GGM}, that the origin is also a monodromic singular point of $X^-$\!, because the origin has not characteristics directions. By Theorem \ref{the:01}~(iii), the origin is a $\Sigma$-monodromic singular point of $X$, and the trajectories of $X$ are oriented counterclockwise sense. See Figure~\ref{elem-degen}.
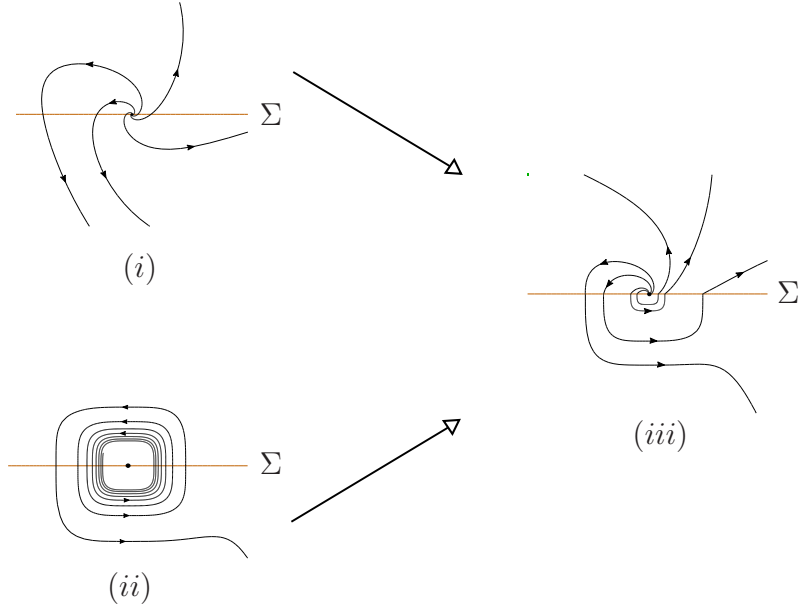
\begin{figure}[h]
	\begin{center}		
		\begin{overpic}[width=0.80\textwidth]{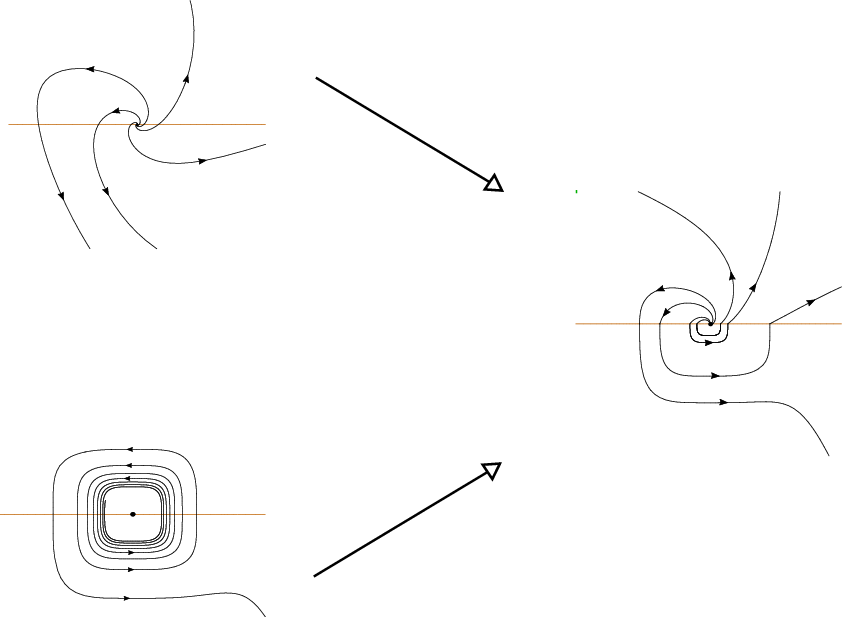}
					%\begin{overpic}[grid,width=0.95\textwidth]{DescontElemetDegenerate.eps}
			\put(15,37) {$(i)$}
			\put(13,-5) {$(ii)$}
			\put(82,15){$(iii)$}
			\put(33,57){$\Sigma$}
			\put(33,11){$\Sigma$}
			\put(101,33.5){$\Sigma$}
		\end{overpic}
	\end{center}
	\vspace{0.7cm}
	\caption{Local phase portraits at the origin. $(i)$ An unstable focus for $X^+$ with $a=b=c=1$.  $(ii)$ An unstable weak focus for $X^-$ with $d=1$.  $(iii)$ The Elementary-Degenerate case for $X$ with $a=b=c=d=1$.}
	\label{elem-degen}
\end{figure}
 
 Doing the change of variables $(x,y)=(\rho\cos\theta,\rho\sin\theta)$ in both  half-planes $\Sigma\cup\Sigma^+$ and $\Sigma\cup\Sigma^-$ the differential systems associated with vector fields $X^+$ and $X^-$ are
\begin{equation}
\label{eqEpolar}
\begin{array}{lcl}
\dfrac{d\rho}{d\theta} &=&\dfrac{f_1^+(\theta) \rho+f_2^+(\theta)\rho^2}{g_0^+ (\theta)+g_1^+ (\theta) \rho}
\\ \\
&=&\dfrac{f_1^+(\theta)}{g_0^+(\theta)}\rho+\dfrac{f_2^+(\theta)g_0^+(\theta)-f_1^+(\theta)g_1^+(\theta)}{g_0^+(\theta)^2}\rho^2\\  \\& & +\dfrac{g_1^+(\theta)(f_1^+(\theta)g_1^+(\theta)-f_2^+(\theta)g_0^+(\theta))}{g_0^+(\theta)^3}\rho^3+{\mathcal O}(\rho^{4})
\end{array}
\end{equation}
and
\begin{equation}
\label{eqDpolar}
\begin{array}{lcl}
\dfrac{d\rho}{d\theta} &=&\dfrac{f_1^-(\theta) \rho+f_3^-(\theta)\rho^3}{g_0^- (\theta)+g_2^- (\theta) \rho^2}
\\ \\
&=&\dfrac{f_1^-(\theta)}{g_0^-(\theta)}\rho+\dfrac{f_3^-(\theta)g_0^-(\theta)-f_1^-(\theta)g_2^-(\theta)}{g_0^-(\theta)^2}\rho^3+{\mathcal O}(\rho^{5}), 
\end{array}
\end{equation}
respectively, where $f_1^+(\theta)=a$, $f_2^+(\theta)=c\cos^3\theta$, $g_0^+(\theta)=b$, $g_1^+(\theta)=-c\cos^2\theta\sin\theta$, $f_1^-(\theta)=\cos\theta\sin\theta(2\cos^2\theta-1)$, $f_3^-(\theta)=d\cos^4\theta\sin^2\theta$, $g_0^-(\theta)=\cos^4\theta+\sin^4\theta$ and $g_2^-(\theta)=d\sin\theta\cos^5\theta$.

The piecewise vector field $X$ satisfies the Hypothesis~{\bf H}. Let  $\rho^\pm(\theta,\rho_0)$ be the analytical solution  of equations \eqref{eqEpolar} and \eqref{eqDpolar} satisfying the initial condition $\rho^\pm(0,\rho_0)=\rho_0$ with $\rho_0> 0$. Following the proof of Proposition~\ref{methpolar}, we obtain the series expansions of $\rho^\pm(\theta,\rho_0)$ in powers of $\rho_0$. Considering $p=q=1$ in \eqref{piplus} and \eqref{piminus} we get $x_0=\rho_0$, and  the displacement function is

\begin{align*}
\Psi(\rho_0)  =&\rho^+(\pi, \rho_0)-\rho^-(-\pi, \rho_0)= \left(\normalfont{\mbox{e}}^{a \pi/b}-1\right)\rho_0 \\& -\dfrac{4ab^2c \normalfont{\mbox{ e}}^{a \pi/b}\left(\normalfont{\mbox{e}}^{a \pi/b}+1\right)}{(a^2+9b^2)(a^2+b^2)}\rho_0^2\\
& +\left(u_3^+(\pi)-\dfrac{d}{\sqrt{2}}\int_0^{-\pi}\frac{(\cos4\theta-1)^2}{(\cos4\theta+3)^2\sqrt{2\cos4\theta+6}}d\theta\right)\rho_0^3\\
&+ {\mathcal O}(\rho_0^{4}).
\end{align*}
We omitted the expression of $u_3^+(\pi)$ for the sake of simplicity and we have that $u_3^+(\pi)=0$ if $a=0$. Thus, for $\rho_0>0$ small enough and $a<0$ we have $\Psi(\rho_0)<0$, i.e. $\rho^+(\pi, \rho_0) < \rho^-(\pi, \rho_0)$, and so the origin is a stable $\Sigma$-focus. On the other hand, if $a>0$ then the origin is an unstable  $\Sigma$-focus.

 If $a=0$ and $d>0$, as \[
\int_0^{-\pi}\frac{(\cos4\theta-1)^2}{(\cos4\theta+3)^2\sqrt{2\cos4\theta+6}}d\theta> 0,\]
 the origin is a stable $\Sigma$-focus and if $d<0$ it is an unstable one. To conclude our analysis, it remains to show that for $a=0$ and $d=0$ we have a $\Sigma$-center. This is true because the differential system \eqref{E} with $a=0$ is invariant by the change of variable $(x, y, t)\mapsto (-x, y, -t)$, i.e., \eqref{E} is reversible, and \eqref{D} with $d=0$ is Hamiltonian with energy $H^+(x,y)= x^4/4+y^4/4$. Note that the points $(x_0,0)$ and $(-x_0,0)$ are in the same level set of $H^+$. This implies that $\Psi(\rho_0)\equiv 0$ and finally we conclude that the origin has a neighborhood free of limit cycles.
%\end{application}

\medskip

\section{Declaration of competing interest}

The authors declare that they have no known competing financial interests or personal
relationships that could have appeared to influence the work reported in this paper

\medskip

\section{Acknowledgments}

The first and the second authors are partially supported by S\~ao Paulo Research Foundation (FAPESP) grant 19/10269-3. The second author also is partially supported by S\~ao Paulo Research Foundation (FAPESP) grant 18/19726-5.
The third author is partially supported by PRONEX/CNPq/FAPEG, by CNPq
grant numbers 475623/2013-4 and 306615/2012-6. All authors are partially supported by CAPES
grant  PROCAD 88881.068462/2014-01.

\addcontentsline{toc}{chapter}{Bibliografia}
%\bibliographystyle{siam}
%\bibliography{cent_discontinuous.bib}

\end{document}